\documentclass[10pt, fleqn]{article}
\usepackage[utf8]{inputenc}
\usepackage[T1]{fontenc}

\usepackage{geometry}
\usepackage{amsmath}
\usepackage{amsthm}
\usepackage{amsfonts}
\usepackage{amssymb}
\usepackage{stmaryrd}
\usepackage{latexsym}
\usepackage{graphicx}
\usepackage{float} 
\usepackage{xcolor}
\usepackage{mathpazo}
\usepackage{euler}
\usepackage{mathdots}
\usepackage{mathtools}
\usepackage[all]{xy}
\usepackage[colorlinks, breaklinks, naturalnames]{hyperref}
\usepackage[initials, lite, msc-links, nobysame]{amsrefs}

\usepackage{leading}
\leading{13pt}

\usepackage{enumitem}
\setitemize{noitemsep, topsep=0pt, listparindent=\parindent, parsep=0pt}
\setlist{listparindent=\parindent, parsep=0pt, topsep=0pt, itemsep=0pt}

\newtheorem{Theorem}{Theorem}[section]
\newtheorem{Corollary}[Theorem]{Corollary}
\newtheorem{Lemma}[Theorem]{Lemma}
\newtheorem{Proposition}[Theorem]{Proposition}

\theoremstyle{definition}

\theoremstyle{remark}
\newtheorem{Remark}[Theorem]{Remark}

\allowdisplaybreaks

\newcommand\ZZ{{\mathbb{Z}}}
\newcommand\NN{{\mathbb{N}}}

\newcommand\U{\mathscr{U}}
\newcommand\N{\mathcal{N}}
\renewcommand\P{\mathfrak{p}}
\newcommand\Q{\mathfrak{q}}
\renewcommand\S{\mathcal{S}}
\newcommand\X{\mathscr{X}}
\newcommand\Y{\mathscr{Y}}

\DeclareMathOperator\lcm{lcm}
\DeclareMathOperator\Aut{Aut}
\DeclareMathOperator\coker{coker}

\DeclareMathOperator\Tor{Tor}
\DeclareMathOperator\Tot{Tot}
\DeclareMathOperator\gldim{gldim}
\DeclareMathOperator\pdim{pdim}
\DeclareMathOperator\im{im}
\newcommand\Id{\mathrm{Id}}
\newcommand\op{\mathrm{op}}
\newcommand\tsum{{\textstyle\sum}}
\newcommand\tint{{\textstyle\int}}
\newcommand\wt[1]{{(#1)}}
\newcommand\q[2]{{[#1]_{#2}}}
\def\place{\mathord-}

\newcommand{\oY}{\overline{Y}}
\newcommand{\oH}{\overline{H}}
\newcommand{\oX}{\overline{X}}
\newcommand{\oB}{\overline{B}}
\newcommand{\oV}{\overline{V}}

\newcommand{\YH}{Y \wedge H}
\newcommand{\YX}{Y \wedge X}
\newcommand{\HX}{H \wedge X}
\newcommand{\YHX}{Y \wedge H \wedge X}

\newcommand{\oYH}{\oY \wedge \oH}
\newcommand{\oYX}{\oY \wedge \oX}
\newcommand{\oHX}{\oH \wedge \oX}
\newcommand{\oYHX}{\oY \wedge \oH \wedge \oX}

\newcommand{\hY}{\hat{Y}}
\newcommand{\hH}{\hat{H}}
\newcommand{\hX}{\hat{X}}

\newcommand{\hV}{\hat{V}}

\newcommand{\hYH}{\hY \wedge \hH}
\newcommand{\hYX}{\hY \wedge \hX}
\newcommand{\hHX}{\hH \wedge \hX}
\newcommand{\hYHX}{\hY \wedge \hH \wedge \hX}

\begin{document}

\title{Hochschild homology and cohomology\\ 
of Generalized Weyl algebras: the quantum case}

\author{Andrea Solotar \and Mariano Su\'arez-Alvarez \and Quimey Vivas
\thanks{This work has been supported by the projects  UBACYTX212,
PIP-CONICET 112-200801-00487, PICT-2007-02182, UBACYT 20020090300102 IJ and
MATHAMSUD-NOCOMALRET. The first and second authors are  research members of
CONICET (Argentina). A.~Solotar thanks Universidad de Valpara\'iso (Project
MECESUP UVA0806).}}

\date{June 27, 2011}

\maketitle

\begin{abstract}
We determine the Hochschild homology and cohomology of the generalized Weyl
algebras of rank one which are of `quantum' type in all but a few
exceptional cases. 

\medskip

\end{abstract}

\section{Introduction}

The Hochschild cohomology $HH^*(A)$ and homology $HH_*(A)$ of a $k$-algebra
$A$ are invariants which are usually hard to compute. For a long time it
has been known that they are related to the smoothness of the algebra. For
example, if $A$ is a commutative algebra $A$ essentially of finite type
---~{\em i.e.}, a quotient of a polynomial algebra on a finite number of
variables by an ideal, or a localization of one of these algebras~---
several authors \cite{A-V} \cite{BACH} \cite{HKR} \cite{Ro} \cite{Ro2} have
obtained results which can be summarized in the statement
  \[
  \text{\itshape If $k$ is a field, $\gldim(A) < \infty$ if and only if there
  exists $n$ such that $HH_i(A)=0$, for all $i>n$.}
  \]
Some years ago, L.\,Avramov and S.\,Iyengar \cite{A-I} proved a cohomological
version of this property:
  \[
  \text{\itshape if $k$ is a field, $\gldim(A) < \infty$ if and only if
  there exists $n$ such that $HH^i(A)=0$, for all $i>n$.}
  \]
The non commutative case is different. After D.\,Happel asked in \cite{Ha}
  \[
  \begin{minipage}{0.9\displaywidth}
  \itshape given a finite dimensional $k$-algebra $A$, is it true that the
  vanishing of $HH^i(A)$ for all large $i$ implies that
  $\gldim(A)< \infty$?
  \end{minipage}
  \]
several articles have been devoted to trying to provide an affirmative
answer. However, in \cite{BGMS} a counterexample was given, 
the ``small'' algebra $k\langle x,y\rangle/(x^2, y^2, xy+qyx)$, with $q\in
k^*$ not a root of unity. Subsequently, Y.\,Han \cite{Han} showed that the
Hochschild homology of this algebra does not vanish in infinitely many
degrees, proposing thus what is now known as \emph{Han's conjecture}:
  \[
  \begin{minipage}{0.9\displaywidth}
  \itshape If all the higher Hochschild homology groups of a finite
  dimensional algebra vanish, then the global dimension of the algebra is
  finite.
  \end{minipage}
  \]
This conjecture has been proved to be true 
for commutative algebras essentially of finite type, not necessarily finite dimensional 
\cite{A-V, BACH},
for finite dimensional graded local algebras \cite{B-M},
for finite dimensional monomial algebras \cite{Han},
for finite dimensional graded cellular algebras in characteristic zero \cite{B-M},
for finite dimensional Koszul algebras in characteristic zero \cite{B-M}, 
for  quantum complete intersections \cite{B-E}, 
for  finite dimensional graded local algebras satisfying the hypotheses of Theorem II of \cite{S-VP},
and for algebras satisfying the hypotheses of Theorem I of \cite{S-VP}.

The general answer is, however, still unknown. The proof of this last case
makes use of the fact that Hochschild homology is functorial, which is not
valid for Hochschild cohomology. The results of Theorem I of \cite{S-VP}
led us to consider the conjecture without the hypothesis of $A$ being
finite dimensional.

It is worth to notice that the proof of the conjecture ---\,homological and
cohomological\,--- in the commutative case, uses the existence of a {\em
model}, that is, a differential graded algebra quasi-isomorphic to the
inital one, and having thus isomorphic Hochschild homology and cohomology.
The importance of the model, stated informally, is that it allows, in a
certain way, to treat more easily the singularities of the algebra. In
other words, the difficulty is no longer in the algebra itself, but in the
differentials of the model. This kind of model, coming from algebraic
topology, always exists in the commutative essentially of finite type case,
but usually not in the non commutative case. One example of a situation
where it exists is treated in Theorem II of \cite{S-VP}. Also, for Koszul
algebras, it is clear that the complex which can be used to compute
Hochschild (co)homology is similar to the one constructed from a model in
the commutative case. So, in our opinion, and although the methods used in
\cite{B-M}, \cite{Han}, \cite{B-E} are different, Han's conjecture has been
proven, up to now, for algebras which have some kind of ``model''.

Following this point of view, in this article we prove it for a family of
non commutative algebras $A_q$, the quantum {\em generalized Weyl
algebras}, which we shall call simply {\em Bavula algebras}. For this we
compute their Hochschild cohomology and homology, completing in this way
the results of \cite{FSS}, leaving out only a few cases. We get the
following two results:

\begin{Theorem}
Let $A=A(\sigma_q,a)$ be a Bavula algebra with $q\in k^\times$ not a root of $1$. Then
  \begin{gather*}
  HH_p(A)=\begin{dcases}
          k^N\oplus \bigoplus_{r\in\ZZ\setminus0} k & \text{if $p=0$;}\\
          k^M\oplus \bigoplus_{r\in\ZZ\setminus0} k & \text{if $p=1$;}\\
          k^M &\text{if $p\geq 2$;}
         \end{dcases}
         \\
  HH^p(A)=\begin{cases}
          k& \text{if $p=0,1$;}\\
          k^N & \text{if $p=2$;}\\
          k^M &\text{if $p\geq 3$;}
         \end{cases}
  \end{gather*}
where $N=\deg a$ and $M=\deg(a:a')$.
\end{Theorem}

\begin{Theorem}
Let $A=A(\sigma_q,a)$ be a Bavula algebra with $q\in k^\times$ such that $q^e=1$. Then
  \begin{gather*}
  HH_p(A)=\begin{dcases}
          k^{\eta(a)} \oplus \bigoplus_{r\in\ZZ\setminus0} \S & \text{if $p=0$;}\\
          k^{\eta(c)} \oplus \bigoplus_{r\in \ZZ\setminus e\ZZ} \bigl(k[h]/(h)\bigr)
            \oplus \bigoplus_{r\in e\ZZ} \S^2
                & \text{if $p=1$;}\\
          k[h]/(c) \oplus \bigoplus_{r\in e\ZZ}\S 
                & \text{if $p=2$;}\\
          k[h]/(c) &\text{if $p\geq 3$.}
         \end{dcases}
         \\
  HH^p(A)=\begin{dcases}
          \bigoplus_{r\in e\ZZ}\S & \text{if $p=0$;}\\
          \bigoplus_{r\in e\ZZ}\S^2 & \text{if $p=1$;}\\
          k^{\eta(a/c)} \oplus k[h]/c \oplus \bigoplus_{r\in e\ZZ}\S &\text{if $p=2$;}\\
          k[h]/c &\text{if $p\geq 3$.}
         \end{dcases}
  \end{gather*}
where, for a polynomial $f\in k[h]$, we write $\eta(f)=\deg f-\tfrac1e\deg
\N(f)$ with $\N$ the operator defined in section~\ref{notaciones} below,
$N=\deg a$, $c=(a:a')$ and and $M=\deg c$.
\end{Theorem}

Whether the `quantum parameter' $q$ appearing in the definition of these
Bavula algebras is a root of unity or not is a fact that plays a fundamental
role, since the computations differ substantially in both cases.

\medskip

The article is organized as follows. In Section~\ref{notaciones} we fix the
notations and state some auxiliar results that will be necessary in the
rest of the article. In Section~\ref{gldim} we recall form \cite{Bav} the
definition of these algebras and we study their global dimension. In
Section~\ref{resolucion} we compute a projective resolution of our algebra
$A$ as an $A$-bimodule. In Section~\ref{homology} we compute the Hochschild
homology and, finally, in Section~\ref{cohomology} we compute the
Hochschild cohomology.

\section{Notations and some generalities}\label{notaciones}

Let $k$ be a field of characteristic zero. If $\lambda\in k$ and $n\geq0$,
we write  $\q{n}{\lambda} = 1 + \lambda  + \cdots + \lambda^{n-1}$;
in~particular, if $\lambda=1$, then $\q{n}{\lambda}=n$.

We fix a scalar $q\in k\setminus\{0,1\}$ and a monic polynomial
$a=\sum_{i=0}^N\alpha_ih^i\in k[h]$ of degree $\deg a=N>1$. Throughout the
paper, $A=A(a,q)$ will denote the $k$-algebra freely generated by letters
$y$, $h$ and $x$ subject to the relations
  \begin{align*}
  &xh=qhx,             
  && yx=a(h), 
  &&hy=qyh,             
  && xy=a(qh).
  \end{align*}
It is easy to see that the set $\{y^ih^j : i,j\geq 0\}\cup \{h^jx^k : j\geq
0,k\geq 1\}$ is a basis of $A$ as a $k$-module. 

We let $\sigma=\sigma_q:k[h]\to k[h]$ be the algebra automorphism such that
$\sigma(h)=qh$. Then $xr=\sigma(r)x$ and $ry=y\sigma(r)$ for all $r\in
k[h]$, and $xy=\sigma(a)$. Moreover, the algebra $A$ is $\ZZ$-graded in
such a way that the generators have degrees $|y|=1$, $|h|=0$ and $|x|=-1$;
we refer to the degree of an element homogeneous with respect to this
grading as its \emph{weight}.

We remark that there is an algebra isomorphism $\Phi:A(a,q)\to
A(\sigma_q(a),q^{-1})$ such that $\Phi(x)=y$, $\Phi(y)=x$ and $\Phi(h)=h$.
This isomorphism maps the homogeneous component of weight $r\in\ZZ$
of~$A(a,q)$ to the component of weight $-r$ of its codomain. This
observation will allow us to carry out homological computations just in
weights $r\ge 0$, since all arguments will be transferable to negative
degrees using $\Phi$.

\bigskip

Given polynomials $p$, $t \in k[h]$, we shall write  $(p:t)$ their greatest
common divisor and $p'$ the derivative of $p$ and we make the convention
that the degree of the zero polynomial is $-\infty$.

We let $c=(a:a')$ and $M=\deg(c)$.  If $q$ is a root of $1$, we let $e$ be
its order; if $q$ is not a root of unity we let $e=0$. If $r\in\ZZ$, we say
that $r$ is \emph{singular} if $e\mid r$, and that it is \emph{regular}
otherwise.

The subring of $k[h]$ fixed by $\sigma$ is $S=\ker(\sigma-1)$, generated by
$h^e$. We say that a polynomial $p\in k[h]$ is \textit{singular} if $p\in
\S$. More generally, when $e>0$ we have $\ker(\sigma-q^l)=h^l\,k[h^e]$ for
each $l\in\{0,\dots,e-1\}$.

If $e>0$, for each $f\in k[h]$ such that $f(0)\neq 0$ we define 
  \[
  \N(f)
        = \lcm\bigl(f:\sigma(f):\cdots:\sigma^{e-1}(f)\bigr)
  \qquad\text{and}\qquad
  \overline{f}
        =\frac{\N(f)}{f}.
  \]
Clearly $\sigma(\N(f))$ is a scalar multiple of $\N(f)$; evaluating both
at~$0$ shows then they are in fact equal, so that $\N(f)\in \S$. The reason
which motivates our interest in the operator $\N$ is the following
proposition:

\begin{Proposition}\label{prop:barra}
Let $f$,~$g\in k[h]$ and suppose $f(0)\neq0$.
\begin{enumerate}[label=\emph{(\roman*)}]

\item If $fg\in \S$, then $\overline{f}\mid g$.

\item If $g\in \S$ and $f\mid g$, then there exists $s\in\S$ such that
$g=\N(f)s$.

\end{enumerate}
\end{Proposition}

\begin{proof}
Since $fg\in \S$, we know that $\sigma^i(fg)=fg$, so $\sigma^i(f)|fg$ for
all $i$. The first statement follows now from the definition of $\N(f)$.
The second one is an immediate consequence.
\end{proof}

We end this section with two technical lemmas which will be of use in the
computation of Sections~\ref{homology} and~\ref{cohomology}.

\begin{Lemma}
\label{lema:piS}
Let $f\in k[h]$ and suppose that $f(0)\neq0$ and that $q$ is a root of
unity of order~$e$. If $\pi:k[h]\to k[h]/(f)$ be the canonical projection,
then for each $l\geq0$ we have
  \[
  \dim\pi(h^lS) = \frac{\deg\N(f)}{e}.
  \]
\end{Lemma}

\begin{proof}
Since multiplication by $\pi(h)$ on $k[h]/(f)$ is an isomorphism, it is
enough to prove this when $l=0$. Let us consider the following commutative
diagram, in which the morphisms are the obvious ones:
  \[
   \xymatrix{
    k[h] \ar[d]_{\pi}\ar@{->>}[r]^-{\pi'} 
        & k[h]/(\N(f)) \ar[dl]^{\rho} \\
    k[h]/(f) 
  }
  \]
Let $G$ be a cyclic group of order $e$ generated by an element $g\in G$. We
endow $k[h]$ with the action of $G$ such that $g$ acts as~$\sigma$. Since
$\N(f)$ is $G$-invariant there is an induced action on $k[h]/(\N(f))$. The
map $\pi'$ is surjective, so the restriction $(\pi')^G:\S \rightarrow
\bigl(k[h]/(\N(f))\bigr)^G$ is surjective too.

The situation is described by the following commutative diagram
  \[
  \xymatrix{
    \S \ar[d]_{\pi|_{\S}}\ar[r]^-{(\pi')^G} 
        & (k[h]/(\N(f)))^G \ar[dl]^{\rho} \\
    k[h]/(f)    
  }
  \]
If $s\in\S$ is such that $\pi|_{\S}(s)=0$, then there exists a $b\in k[h]$
such that $fb=s\in\S$ and it follows from the previous proposotion that
$b=\bar fs_1$ for some $s_1\in\S$: we see that $s=\N(f)s_1$ and
$(\pi')^G(s)=0$. As $(\pi')^G$ is surjective, this implies that the map
$\rho$ is injective and, as a consequnce, that $\dim\pi(S)=\dim
(k[h]/(\N(f)))^G$.  

Now, $k[h]/(\N(f))$ has $\{h^i:0\leq i<\deg\N(f)\}$ as a basis and the
action of $G$ is diagonal with respect to it. It is immediate, then, that
$\dim (k[h]/(\N(f)))^G=\tfrac1e\deg\N(f)$
\end{proof}

\begin{Lemma}\label{lema:coker}
Let $f\in k[h]$ such that $f(0)\neq0$, $q\in k$ a root of unity of
order~$e$, $l\geq0$ and consider the $\S$-linear map $\psi_{f,l}:p\in
k[h]\mapsto (\sigma-q^l)(fp)\in k[h]$. Then
  \[
  \coker \psi_{f,l} \cong h^l\S \oplus k^{\eta(f)}
  \]
with $\eta(f)=\deg f-\tfrac1e\deg\N(f)$.
\end{Lemma}

\begin{proof}
We decompose $k[h]\cong \S \oplus h\S \cdots \oplus h^{e-1}\S$ as
$\S$-module. Since $\ker(\sigma-q^l)=h^l\S$, the map $\sigma-q^l$ induces
an injective map $k[h]/h^l\S\rightarrow k[h]$, still denoted $\sigma-q^l$.
Consider the following diagram
  \[
  \xymatrix{
     k[h] \ar[r]^-f 
        & k[h] \ar[r]^{\sigma-q^l} \ar[d]_-p
        & k[h] \\ 
        & k[h]/h^l\S \ar[ur]_-{\sigma-q^l}
  }
  \]
Because $\sigma-q^l$ is injective, it is immediate that $\coker
\psi_{f,l}\cong h^l\S\oplus\coker(p\circ f)$ and, since
  \[
  \coker p\circ f 
        \cong  \frac{k[h]}{h^l\S+(f)}
        \cong  \frac{k[h]}{(f)} \Bigm/  \pi(h^lS)
  \]
with $\pi$ the map defined in Lemma~\ref{lema:piS}, we see that $\dim\coker
p\circ f=\eta(f)$.
\end{proof}

We remark that the isomorphism in the statement of this lemma is actually
an isomorphism of $\S$-modules, if we identify the summand $k^{\eta(f)}$
with the quotient ${k[h]}/(h^l\S+(f))$ appearing in the proof.

\section{Global dimension}\label{gldim}

Given a noetherian algebra $R$ which is an integral domain, a non zero
central element $a \in R$ and an algebra automorphism $\sigma \in \Aut_k(R)$,
the \emph{Bavula algebra} $\Lambda=\Lambda(R, \sigma, a)$ is the $k$-algebra
generated by $R$ and two variables $x$, $y$ subject to the relations
 \begin{align*}
  &yx=a, 
  &&xy=\sigma(a), 
  &&xr=\sigma(r)x, 
  &&ry=y\sigma(r)
  \end{align*}
for all $r\in R$. It was introduced by V.~Bavula in~\cite{Bav} with the
name of \emph{generalized Weyl algebra}. The algebra $A$
introduced in Section~\ref{notaciones} is a special case of this
construction.

The algebra~$A$ is a noetherian domain and there is a $\ZZ$-grading
on~$\Lambda$ with all elements of~$R$ in degree~$0$, and $x$ and $y$ in
degrees $-1$ and~$1$, respectively; we denote $|u|$ the degree of an
homogeneous element $u\in\Lambda$ and call it its \emph{weight}.

\bigskip

Using the easily obtained description of automorphisms of $k[h]$, one can
obtain the following classification of the algebras of the form
$\Lambda(k[h],\sigma,a)$ up to isomorphism, as in \cite{RS}:

\begin{Proposition} \label{prop1}
The algebra $\Lambda=\Lambda(k[h],\sigma, a)$ is isomorphic to exactly one of the
following list:
\begin{enumerate}

\item $\Lambda(k[h],\Id,a)$ for some $a\in k[h]$;

\item $\Lambda(k[h],\sigma_{cl},a)$  with $\sigma_{cl}(h)=h-1$ and  $a\in k[h]$;

\item $\Lambda(k[h],\sigma_{q},a)$  with $q\in k\setminus\{0,1\}$, $\sigma_q(h)=qh$ and $a\in
k[h]$.

\end{enumerate}
We refer to case 2 as the \emph{classical case} and to case 3 as the
\emph{quantum case}.
\end{Proposition}

\bigskip

If $b\in R$, let $I(x,b)=\Lambda x+\Lambda b\subseteq \Lambda$. Bavula proved in \cite{Bav2}
the following result concerning the global dimension of his algebras:

\begin{Theorem}\cite[Thm. 3.5]{Bav2}
\label{posta}
If $R$ is a commutative Noetherian domain of finite global dimension $n$
and $a\neq 0$, then the following two conditions are equivalent:
\begin{itemize}

\item $\gldim\Lambda<\infty$ 

\item $\pdim_\Lambda \Lambda/I(x,\P)<\infty$ for all prime ideals $\P$ of
$R$ which contain $a$. \qed

\end{itemize}
\end{Theorem}

When $R=k[h]$, the hypotheses of this theorem are satisfied and we can give
a characterization of Bavula algebras of finite global dimension. 

\begin{Theorem}
Let $R=k[h]$,  $a \in R$, $\sigma\in\Aut_k(R)$ and $\Lambda=\Lambda(R,
\sigma, a)$. Then
  \[ 
  \gldim\Lambda < \infty \iff (a:a')=1. 
  \]
\end{Theorem}

\begin{proof}
The ``only if'' part has been proved by Bavula in \cite{Bav2}, so we only
have to prove the ``if'' part.

Let $p \in R$ be a prime element which divides $a$, so that there is a
$b\in R$ with $a=pb$. The canonical short exact sequence of left
$\Lambda$-modules
  \[ 
  0 \to I(x,p) \to \Lambda \to \Lambda/I(x, p) \to 0 \label{eq:split} 
  \]
tells us that $\pdim_\Lambda\Lambda/I(x,p)< \pdim_\Lambda I(x,p) + 2$.  We
shall prove that if $(a:a')=1$, then $I(x,p)$ is a projective
$\Lambda$-module.

We start by showting that $\Lambda x\cap \Lambda p= I(x,b)p$.  Fix $f\in \Lambda
x\cap \Lambda p$; we may assume that $f$ is homogeneous with respect to the
weight and that $|f|=r\geq0$: the case in which the weight of $f$ is
negative is similar. Since $f\in \Lambda x\cap \Lambda p$, there exist $u$,
$v\in R$ such that $f=y^{r+1}ux=y^rvp$. As
\[
y^{r+1}ux=y^ra\sigma^{-1}(u)= y^rpb\sigma^{-1}(u)= y^r\sigma^{-1}(u)bp
\]
and $\Lambda$ is a domain,  $\sigma^{-1}(u)b=v$ and, in consequence, $f\in
I(x,b)p$.  The other inclusion is easy.

Consider now the short sequence of left $\Lambda$-modules 
\begin{equation}\label{eq:split} 
\xymatrix{
0  \ar[r]
  &   I(x,b)  \ar[r]^-{\gamma}  
  & \Lambda\oplus \Lambda \ar[r]^-{\phi}  
  & I(x, p)  \ar[r]
  & 0
}
\end{equation} 
where $\phi(\alpha, \beta)= \alpha x - \beta p$ and $\gamma(w)=(wpx^{-1}, w)$; this 
last expression makes sense because for all $p\in I(x,b)$ we have $wp\in Ax=xA$ and 
$A$ is a domain.

It is clear that $\gamma$ is a monomorphism, $\phi$ is an epimorphism and
that $\im\gamma\subseteq \ker \phi$. The sequence~\eqref{eq:split} is in
fact exact: to check the other inclusion suppose that $(\alpha, \beta)\in
\Lambda\oplus \Lambda$ is such that $\alpha x=\beta p$. This element
belongs to $\Lambda x\cap \Lambda p=I(x,b)p$, and it follows that $\alpha=
\beta p x^{-1}$. If $(a:a')=1$, then $(p:b)=1$ and there exist $s$, $t \in
R$ such that $1=sp+tb$. We define the map $\psi: \Lambda\oplus \Lambda \to
\Lambda$ by $\psi(\alpha, \beta)=\alpha xs + \beta bt$. It is easy to
verify that $\im \psi \subseteq I(x,b)$ and that $\psi\circ
\gamma=\Id_{I(x,b)}$. As a consequence, the sequence \eqref{eq:split}
splits and $I(x,p)$ is a projective $\Lambda$-module. \end{proof}

In particular, for the algebras introduced in Section~\ref{notaciones} we
have the following:

\begin{Corollary}
For all $q\in k\setminus\{0,1\}$ and all $a\in k[h]$ we have
  \[
  \gldim\Lambda(k[h],\sigma_q,a)<\infty
        \implies
        \gldim\Lambda(k[h],\sigma_q,a)=2.
  \]
\end{Corollary}

\begin{proof}
It follows from~\cite[Thm. 2.7]{Bav2} that if the global dimension
of~$\Lambda(R,\sigma,a)$ is finite, it equals either~$\gldim R$ or~$\gldim
R+1$. In the situation of the corollary, then,
$\gldim\Lambda(k[h],\sigma,a)\in\{1,2\}$ if it is finite. Moreover, using
\cite[Thm. 3.7]{Bav2}, we see that $\gldim\Lambda(k[h],\sigma,a)=2$ if and
only if either \emph{(i)} there is a maximal ideal of $k[h]$ of height $1$
with finite orbit under $\sigma$, or \emph{(ii)} if there  are maximal
ideals $\P$, $\Q$ of $k[h]$ of height $1$ such that $\sigma^i(\P)= \Q$ for
some $i\neq 0$, $i\in \ZZ$ and $a\in \P\cap\Q$.  Since the ideal $(h)$ of
$k[h]$ is obviously fixed by $\sigma$ and it is of height $1$, we are
always in case \emph{(i)}, and the corollary follows.
\end{proof}

The conditions \emph{(i)} and \emph{(ii)} mentioned in the proof of this corollary are not exclusive. Indeed, most of the complication encountered in the computations that follow arises when the algebra $A$ satisfies condition \emph{(ii)} or, in other words, when the polynomial $a$ has two roots in the same orbit under $\sigma_q$.

\section{A projective resolution}\label{resolucion}

The purpose of this section is to construct a projective resolution of the
Bavula algebra $A$. We do this in two steps, using an algebra $B_l$ as an
intermediate step, as in \cite{FSS}.

\subsection{Smith algebras}

Fix a polynomial $l=\sum_{i=0}^m \lambda_i H^i\in k[H]$, with $m>0$ and
$\lambda_m \neq 0$. We consider the $k$-algebra $B_l$, or simply $B$, with
generators $Y$, $H$ and $X$ subject to the relations
  \begin{align*}
  &HY=qYH, 
  &&[X,Y]=l, 
  &&XH=qHX.
  \end{align*}
This algebra was considered by P.\,Smith in~\cite{smith}, observing that it
is in many aspects similar to the enveloping algebra $\U(\mathfrak{sl}_2)$;
we will call it a \emph{Smith algebra}. 

The set $\{Y^iH^jX^k : i,j,k\geq 0\}$ is a basis of $B$ as a
$k$-module. Let $V=kY\oplus kH \oplus kX\subset B$.  Setting $|X|=|Y|=1$
and $|H|=0$ we obtain a grading on $TV$, which induces an increasing
filtration on $B$; let us write $\oY$, $\oH$ and $\oX$ for the principal
symbols of $Y$, $H$ and $X$, respectively, in $\oB=gr\ B$. Then $\oB$ is
the $k$-algebra generated by $\oY$, $\oH$ and $\oX$, subject to the
relations
  \begin{align*}
  &\oH\oY=q\oY\oH,  
  &&[\oX,\oY]=0,
  &&\oX\oH=q\oH\oX.
  \end{align*}
Of course, $\oV\cong gr\ V$ is spanned by $\oX$, $\oY$ and $\oH$, and these
elements are $k$-linearly independent.

We will use frequently the following notation: given a function $f$ of two
integer arguments, and $i \in \NN_0$, we will write
  \[
  \tint_i f(s,t) = \sum_{\substack{s+t+1=i\\ 0\leq s,t}} f(s,t).
  \]
In particular, in such an ''integral'' expression, the indices $s$ and $t$
are \textit{not} free. We note that the identity
  \[
  \tint_i f(s+1,t) - \tint_i f(s,t+1) = f(i,0) - f(0,i)
  \]
holds for all $f$ and $i$: we will make use of it repeatedly.

\bigskip

Consider now the complex of $B^e$-projective modules over $B$ 
  \begin{equation}
  \label{eq:smith_complex}
  \xymatrix{
  0\ar[r]
    & B|\bigwedge^3 V | B \ar[r]
    & B|\bigwedge^2 V | B \ar[r]^-d
    & B| V | B \ar[r]^-d
    & B | B \ar@{->>}[r]^-\mu
    & B 
  }
  \end{equation}
with differentials given by
  \begin{gather*}
  d(1|v|1)=1|v-v|1,\qquad \forall v \in V; \\
  d(1|\HX|1)=1|X|H-qH|X|1-q|H|X+X|H|1; \\
  d(1|\YX|1)=1|X|Y-Y|X|1-1|Y|X+X|Y|1-\tsum_i \tint_i \lambda_i H^s|H|H^t; \\
  d(1|\YH|1)=1|H|Y-qY|H|1-q|Y|H+H|Y|1; \\
  \!\begin{multlined}[t][0.8\displaywidth]
  d(1|\YHX|1)=1|\HX|Y -qY|\HX|1 -q|\YX|H + q H|\YX|1 \\
                                + q|\YH|X - X|\YH|1.
  \end{multlined}
  \end{gather*}
The verification that $d^2=0$ is a routine computation.

The filtrations on $B$ and on $V$ determine a filtration on the complex
\eqref{eq:smith_complex}, whose associated graded complex is
  \[
  \xymatrix@1{
  0\ar[r] 
        & \oB|\bigwedge^3 \oV | \oB\ar[r]^-d 
        & \oB|\bigwedge^2 \oV | \oB\ar[r]^-d 
        & \oB| \oV | \oB \ar[r]^-d 
        & \oB | \oB \ar@{->>}[r]^-{\mu} 
        & \oB
  }
  \]
with $\oB^e$-linear differentials determined by the conditions
  \begin{gather*}
  d(1|v|1)
        = 1|v-v|1,\qquad \forall v \in \oV; \\
  d(1|\oHX|1)
        = 1|\oX|\oH-q\oH|\oX|1-q|\oH|\oX+\oX|\oH|1; \\
  d(1|\oYX|1)
        = 1|\oX|\oY-\oY|\oX|1-1|\oY|\oX+\oX|\oY|1; \\
  d(1|\oYH|1)
        = 1|\oH|\oY-q\oY|\oH|1-q|\oY|\oH+\oH|\oY|1; \\
  \!\begin{multlined}[t][0.8\displaywidth]
  d(1|\oYHX|1)
        = 1|\oHX|\oY -q\oY|\oHX|1 -q|\oYX|\oH + q \oH|\oYX|1 \\
                                + q|\oYH|\oX - \oX|\oYH|1.
  \end{multlined}
  \end{gather*}
This complex is exact. Indeed, there is a left $\oB$-linear contraction given by
  \begin{gather*}
   s(1) 
        = 1|1;\\
   s(1|\oY^i\oH^j\oX^k) 
        = \tsum_i \tint_i \oY^s|\oY|\oY^t\oH^j\oX^k 
        + \tsum_i \tint_j \oY^i\oH^s|\oH|\oH^t\oX^k 
        + \tsum_i \tint_k \oY^i\oH^j\oX^s|\oX|\oX^t; \\
   s(1|\oY|\oY^i\oH^j\oX^k) 
        = 0; \\
   s(1|\oH|\oY^i\oH^j\oX^k) 
        = \tsum_i \tint_i q^s\oY^s|\oYH|\oY^t\oH^j\oX^k; \\
   s(1|\oX|\oY^i\oH^j\oX^k)
        = \tsum_i \tint_i \oY^s|\oYX|\oY^t\oH^j\oX^k 
        + \tsum_i \tint_j q^s \oY^i\oH^s|\oHX|\oH^t\oX^k; \\
   s(1|\oHX|\oY^i\oH^j\oX^k)
        = \tsum_i \tint_i q^s\oY^s|\oYHX|\oY^t\oH^j\oX^k; \\
   s(1|\oYX|\oY^i\oH^j\oX^k)
        = 0; \\
   s(1|\oYH|\oY^i\oH^j\oX^k)
        = 0.
  \end{gather*}
It follows that the complex \eqref{eq:smith_complex} is a $B^e$-projective
resolution of $B$.

\subsection{Bavula algebras}

Next we construct a resolution of our Bavula algebra as a bimodule over
itself. Let $l=\sigma(a) - a$; then $\deg a\ge \deg l$ and $l= \sum_{i=0}^N
\lambda_ih^i$ with $\lambda_i=(q^i-1)\alpha_i$. We consider the Smith
algebra $B=B_l$ corresponding to the polynomial $l$, and the element
$\Omega=YX-a \in B$. A simple computation shows that $\Omega=XY-\sigma(a)$
and that $\Omega$ is central in $B$. In particular, $B\Omega=\Omega B$ is a
two-sided ideal of $B$ and the quotient $B/\Omega B$ is isomorphic to $A$
via an isomorphism which sends the classes of $Y$, $H$ and $X$ to $y$, $h$
and $x$ respectively. We will identify $A$ with the quotient. 

Let $\pi: B \to A$ denote the canonical projection. Since $\Omega$ is not a
zero divisor in $B$, the complex
  \begin{equation}
  \label{eq:resolution1}
  0 \longrightarrow B \overset{\Omega}{\longrightarrow} B \overset{\pi}{\longrightarrow} A \longrightarrow 0
  \end{equation}
is a projective resolution of $A$ as a $B$-module both on the left and on
the right; here the first arrow is simply the multiplication by $\Omega$.
On the other hand, by applying the functor $(\place)\otimes_B A$ to the
resolution \eqref{eq:smith_complex} of $B$ as $B^e$-module given in the
previous subsection, we obtain the complex 
  \begin{equation}
  \label{eq:resolution2}
  \xymatrix{
  0 \ar[r]
        & B|\bigwedge^3 V | A \ar[r]^-d
        & B|\bigwedge^2 V | A \ar[r]^-d
        & B|V|A \ar[r]^-d
        & B|A \ar[r]^-\mu
        & A \ar[r]
        & 0
  }
  \end{equation}
with $B\otimes A^\op$-linear differentials given by
  \begin{gather*}
  d(1|v|1)
        = 1|v-v|1,\qquad \forall v \in V; \\
  d(1|\HX|1)
        = 1|X|h-qH|X|1-q|H|x+X|H|1; \\
  d(1|\YX|1)
        = 1|X|y-Y|X|1-1|Y|x+X|Y|1-\tsum_i \tint_i \lambda_i H^s|H|h^t; \\
  d(1|\YH|1)
        = 1|H|y-qY|H|1-q|Y|h + H|Y|1; \\
  d(1|\YHX|1)
        = 1|\HX|y -qY|\HX|1 -q|\YX|h + q H|\YX|1 + q|\YH|x - X|\YH|1.
  \end{gather*}
The homology of this complex is $\Tor^B_{\bullet}(B,A)$, so that it is in
fact acyclic. This means that \eqref{eq:resolution2} is a projective
resolution of $A$ as a left $B$-module.

\bigskip

There exist morphisms between the two resolutions \eqref{eq:resolution1}
and \eqref{eq:resolution2} of the left $B$-module $A$ lifting the identity map
of $A$:
  \begin{equation} \label{eq:comparacion}
  \xymatrix{
  0 \ar[r] 
        & B \ar[r]^{\Omega} \ar@<0.5ex>[d]^{f_1}
        & B \ar@{->>}[r]^{\pi} \ar[d]^{f_0}
        & A \ar[d]^{1_A} 
        \\
  \cdots \ar[r] 
        & B|V|A \ar[r]^d \ar@<0.5ex>[u]^{g_1}
        & B|A \ar@{->>}[r]^{\mu} \ar@<0.5ex>[u]^{g_0}
        & A \ar@<0.5ex>[u]^{1_A}
  }
  \end{equation}
given by
  \begin{align*}
   &f_0(1)=1|1; 
        && f_1(1)=-Y|X|1-1|Y|x+\tsum_i\tint_i \alpha_i H^s|H|h^t;\\
   &g_0(1|y^ih^j)=Y^iH^j; 
        && g_0(1|h^jx^k)=H^jX^k;\\
   &g_1(1|Y|y^ih^j)=0; 
        && g_1(1|Y|h^jx^{k+1})=-q^{-j}H^jX^k\\
   &g_1(1|H|y^ih^j)=0; 
        && g_1(1|H|h^jx^k)=0;\\
   &g_1(1|X|y^{i+1}h^j)=-Y^iH^j; 
        && g_1(1|X|h^jx^k)=0.
  \end{align*}

Using \eqref{eq:resolution1}, the computation of $\Tor^B_{\bullet}(A,A)$ is immediate because
the only relevant differential vanishes, and we see that
  \begin{equation}
   \label{eq:tor}
   \Tor^B_p(A,A)=\begin{cases}
		  A\otimes_B B, & p=0;\\
		  A\otimes_B B, & p=1;\\
		  0,   &p\geq 2.
		 \end{cases}
  \end{equation}
Since $\Tor^B_{\bullet}(A,A)$ can be calculated from any resolution of $A$
as left $B$-module, the complex obtained by applying the functor
$A\otimes_B(\place)$ to the resolution (\ref{eq:resolution2}), that is
  \begin{equation}
  \label{eq:complex_row}
  \xymatrix@1{
  0 \ar[r] 
        & A|\bigwedge^3 V | A\ar[r]^d 
        & A|\bigwedge^2 V | A\ar[r]^d 
        & A| V | A \ar[r]^d 
        & A | A
  }
  \end{equation}
with $A^e$-linear differential:
  \begin{gather*}
   d(1|v|1)=1|\pi(v)-\pi(v)|1,\qquad \forall v \in V;\\
   d(1|\HX|1)=1|X|h-qh|X|1-q|H|x+x|H|1; \\
   d(1|\YX|1)=1|X|y-y|X|1-1|Y|x+x|Y|1-\tsum_i \tint_i \lambda_i h^s|H|h^t; \\
   d(1|\YH|1)=1|H|y-qy|H|1-q|Y|h + h|Y|1; \\
   d(1|\YHX|1)=1|\HX|y -qy|\HX|1 -q|\YX|h + q h|\YX|1 + q|\YH|x - x|\YH|1,
  \end{gather*}
has homology isomorphic to $\Tor^B_{\bullet}(A,A)$. Using the morphisms
$f_{\bullet}$ and $g_{\bullet}$ from~\eqref{eq:comparacion}, we see that the homology
of the complex (\ref{eq:complex_row}) is freely generated as left
$A$-module by the classes of the cycles $1|1\in A\otimes A$ and
  \[
  y|X|1+1|Y|x-\tsum_i\tint_i \alpha_i h^s|H|h^t \in A\otimes V\otimes A,
  \]
of degrees $0$ and $1$, respectively.

\subsection{The resolution}
\label{resolution}

Next we consider the third-quadrant double complex $X_{\bullet,\bullet}$
depicted in the following diagram
  \[
  \xymatrix{
  {}
        &   
        & 0\ar[r] 
        & A|\bigwedge^3 V | A\ar[r]^-d 
        & A|\bigwedge^2 V | A\ar[r]^-d 
        & A| V | A \ar[r]^-d 
        & A | A 
        \\
  {}
        & 0\ar[r] 
        & A|\bigwedge^3 V | A\ar[r]^-d \ar[u] 
        & A|\bigwedge^2 V | A\ar[r]^-d\ar[u]_-{\delta}  
        & A| V | A \ar[r]^-d \ar[u]_-{\delta} 
        & A | A\ar[u]_-{\delta} 
        \\
  \iddots 
        & \iddots\ar[u] 
        & \iddots\ar[u]_-{\delta} 
        & \iddots\ar[u]_-{\delta} 
        & \iddots \ar[u]_-{\delta}
  }
  \]
so that $X_{p,q}=A|\bigwedge^{p-q}V|A$ if $q\geq 0$ and $X_{p,q}=0$
otherwise, with horizontal $A^e$-linear differentials $d$, of bidegree
$(-1,0)$, given as in~\eqref{eq:complex_row}, and vertical differentials
$\delta$, of bidegree $(0,1)$, given by
  \begin{gather*}
  \delta(1|1)
	= y|X|1+1|Y|x-\tsum_i \tint_i \alpha_i h^s|H|h^t;\\
  \delta(1|Y|1)
	= -y|\YX|1+\tsum_i \tint_i \alpha_i q^t h^s|\YH|h^t;\\
  \delta(1|H|1)
	= 1|\YH|x-y|\HX|1;\\
  \delta(1|X|1)
	= 1|\YX|x-\tsum_i \tint_i \alpha_i q^s h^s|\HX|h^t;\\
  \delta(1|\YH|1)
	= y|\YHX|1;\\
  \delta(1|\YX|1)
	= \tsum_i \tint_i \alpha_i q^{i-1}h^s|\YHX|h^t;\\
  \delta(1|\HX|1)
	= 1|\YHX|x.
  \end{gather*}
A direct computation shows that it is indeed a complex with anti-commuting differentials.

To compute the homology of the total complex $\Tot X_{\bullet,\bullet}$ we
use the spectral sequence $E$ which arises from the filtration by rows. The
differential on the first page $E^0$ of this spectral sequence is the
horizontal differential $d$ on $X_{\bullet,\bullet}$, and we have
essentially computed the corresponding homology in \eqref{eq:tor}: we see
from this that the second page $E^1$ of $E$ is, up to isomorphism, as in
the following diagram:
  \[
  \xymatrix@R-10pt{
  {} 
        & 
        & 
        & 0
        & 0
        & A 
        & A 
        \\
  {}
        & 
        & 0
        & 0
        & A 
        & A\ar[u]_-{d^1}
        \\
  {}
        & 0
        & 0
        & A 
        & A\ar[u]_-{d^1}
        \\
  \iddots
        & \iddots
        & \iddots
        & \iddots
  }
  \]
Consequently, the only components of the differential $d^1$ which can
possibly be non zero are the maps $d^1_{p,p}:E^1_{p,p}\rightarrow
E^1_{p,p-1}$, with $p\geq 1$, and they are induced by the vertical
differentials $\delta$ in~$X_{\bullet,\bullet}$. We know that $E^1_{p,p}$
and $E^1_{p,p-1}$ are free left $A$-modules on the horizontal homology
classes of $1|1\in X_{p,p}$ and $\omega =  y|X|1+1|Y|x-\sum_i\tint_i
\alpha_i h^s|H|h^t \in X_{p,p-1}$, respectively.  In view of the definition
of $\delta$, $d^1([1|1])=[\omega]$, and, since $d^1$ is $A$-linear, this
shows that all components of $d^1$ which are not trivially zero are
isomorphisms.

It follows that the complex $\Tot X_{\bullet,\bullet}$ is acyclic over $A$,
with augmentation given by the multiplication map $\mu:X_{0,0}=A\otimes
A\rightarrow A$ and, since its components are free $A^e$-modules, it is in
fact a projective resolution of $A$ as $A^e$-module.

We consider the grading $V$ such that~$Y$, $H$ and~$X$ are homogeneous of
degrees $1$, $0$ and~$-1$, respectively. This, together with the grading
of~$A$ by weights, induces a grading on the complex $X_{\bullet,\bullet}$
such the differentials are homogeneous. It follows that the complexes
obtained by applying the functors $A\otimes_{A^e}(\place)$ and
$\hom_{A^e}(\place,A)$ below will also be graded by weights in a natural
way.

\section{Hochschild homology}\label{homology}

In this section we will compute the Hochschild homology of $A$ using the
resolution described in the previous section and a spectral sequence argument.

Applying the functor $A\otimes_{A^e}-$ to $X_{\bullet,\bullet}$ and identifying
$A\otimes_{A^e}(A\otimes \wedge^p V \otimes A)$ with $A\otimes\wedge^p V$ in the
natural way, we get a double complex such that the homology of its total complex
is $HH_*(A)$, the Hochschild homology of $A$ with coefficients in itself.
This double complex is
  \[
  \xymatrix@R-5pt{
    &   
    & 0\ar[r] 
    & A|\bigwedge^3 V \ar[r]^d 
    & A|\bigwedge^2 V \ar[r]^d 
    & A| V \ar[r]^d 
    & A \\
    & 0\ar[r] 
    & A|\bigwedge^3 V \ar[r]^d \ar[u]_{\delta} 
    & A|\bigwedge^2 V \ar[r]^d\ar[u]_{\delta}  
    & A| V  \ar[r]^d \ar[u]_{\delta} 
    & A \ar[u]_{\delta}\\
  \iddots 
    & \iddots\ar[u]_{\delta} 
    &\iddots\ar[u]_{\delta}
    &\iddots\ar[u]_{\delta} 
    &\iddots \ar[u]_{\delta}
  }
  \]
with differentials given by
  \begin{subequations}
  \begin{gather}
  d(u|Y)=[y,u], \label{eq:d0-y} \\
  d(u|H)=[h,u],\label{eq:d0-h} \\
  d(u|X)=[x,u],\label{eq:d0-x} \\
  d(u|\YH)=[y,u]_q|H+[u,h]_q|Y, \label{eq:d1-yh} \\
  d(u|\YX)=[y,u]|X+[u,x]|Y-\tsum_i \lambda_i\tint_ih^t u h^s|H, \label{eq:d1-yx} \\
  d(u|\HX)=[h,u]_q|X+[u,x]_q|H, \label{eq:d1-hx}  \\
  d(u|\YHX)=[y,u]_q|\HX+q[u,h]|\YX-[u,x]_q|\YH, \label{eq:d2}
  \end{gather}
  \end{subequations}
and
  \begin{subequations}
  \begin{gather}
  \delta(u)=uy|X+xu|Y-\tsum_i\alpha_i\tint_ih^tuh^s|H, \label{eq:delta-pri} \\
  \delta(u|Y)=-uy|\YX+\tsum_i \alpha_i\tint_iq^th^tuh^s|\YH,  \\
  \delta(u|H)=xu|\YH-uy|\HX, \\
  \delta(u|X)=xu|\YX-\tsum_i \alpha_i\tint_iq^sh^tuh^s|\HX, \\
  \delta(u|\YH)=uy|\YHX,  \\
  \delta(u|\YX)=\tsum_i\alpha_i\tint_i q^{i-1}h^t u h^s|\YHX, \\
  \delta(u|\HX)=xu|\YHX. \label{eq:delta-ult}
  \end{gather}
  \end{subequations}
We will use the filtration by columns on this complex and denote $E$ the
corresponding spectral sequence, which, as the complex
$A\otimes_{A^e}X_{\bullet,\bullet}$ itself, is graded by weights. We are
going to write $HH_\bullet(A)^\wt r$ and $E^\wt r$ the components of weight
$r$ in $HH_\bullet(A)=H(A\otimes_{A^e}X_{\bullet,\bullet})$ and $E$.

\subsection{First Page}

Let $\X$ be the complex 
  \begin{equation}
  \label{colg}
  \xymatrix{
    0 \ar[r]
      & A \ar[r]^-{\delta} 
      & A| V \ar[r]^-{\delta} 
      & A|\bigwedge^2 V \ar[r]^-{\delta} 
      & A|\bigwedge^3 V \ar[r] 
      & 0
  }
  \end{equation}
graded so that $A$ and $A|\bigwedge^3V$ are in degrees $0$ and $3$,
respectively, and with differentials as in
\eqref{eq:delta-pri}--\eqref{eq:delta-ult}. It is clear that $E^1_{p,q} =
H_{p-q}(\X)$ for all $q>0$ and that the $E^1_{p,0}$ can be seen as
cokernels of the differentials of $\X$.

For each $r\in\ZZ$, let $\X^\wt r$ be the homogeneous component of weight
$r$. In this subsection, we compute
$H_\bullet(\X)=\bigoplus_{r\in\ZZ}H_\bullet(\X^\wt r)$.

\begin{Proposition}
If $r\in\ZZ$ is non zero, then the complex $\X^\wt r$ is exact. On the
other hand, there are isomorphisms of $S$-modules
  \[
  H_p(\X^\wt0) \cong
    \begin{cases}
    k[h]/(c), & \text{if $2\leq p\leq 3$;} \\
    0, & \text{otherwise.}
    \end{cases}
  \]
\end{Proposition}

\begin{proof}
One way to organize the computation is as follows:
\begin{itemize}

\item If $u=p\in \X^\wt0_0$, with $p\in k[h]$, then
  \begin{equation} \label{eq:delta0}
  \delta(u)=y\sigma(p)|X+\sigma(p)x|Y-a'p|H.
  \end{equation}
As $A$ is a domain, it follows immediately that $\delta$ is a monomorphism
and that $H_0(\X^\wt0)=0$.

\item Let $u=p_1x|Y+p_2|H+yp_3|X \in \X^\wt0_1$, with $p_1$, $p_2$, $p_3\in
k[h]$. We know that
  \begin{equation} \label{eq:delta1}
  \delta(u) =
        (p_1\sigma(a')+\sigma(p_2))x|\YH 
        + \sigma(a)(p_3-p_1)|\YX 
        - y(p_3\sigma(a')+\sigma(p_2))|\HX.
  \end{equation}
Since $A$ is a domain, we see that $\delta(u)=0$ if and only if $p_1=p_3$
and $p_2=-\sigma^{-1}(p_1)a'$. This description of cyles together with the
expression \eqref{eq:delta0} of boundaries imply that $H_1(\X^\wt0)=0$.

\item Let $u=p_1x|\YH+p_2|\YX+yp_3|\HX \in \X^\wt0_2$. A computation shows
that
  \begin{equation} \label{eq:delta2}
   \delta(u)=(p_1 \sigma(a) + p_2\sigma(a')+\sigma(a)p_3 )|\YHX.
  \end{equation}
Suppose that $u\in\ker\delta$, so $p_1 \sigma(a) +
p_2\sigma(a')+\sigma(a)p_3 =0$. It follows immediately from this that
$\sigma(\tfrac{a}{c})(p_1+p_3)=-\sigma(\tfrac{a'}{c})p_2$. Since $a/c$ and
$a'/c$ are coprime, there exists $g\in k[h]$ such that
$p_1+p_3=-\sigma(\tfrac{a'}{c})g$ and $p_2=\sigma(\tfrac{a}{c})g$. If $v$,
$r\in k[h]$ are such that $g=v\sigma(c)+r$ and $\deg r<\deg c$, then $u$ is
homologous to 
  \[
  u-\delta(\sigma^{-1}(p_1)|H+yv|X) = r\sigma(\tfrac{a}{c})|\YX - y r
  \sigma(\tfrac{a'}{c})|\HX.
  \]
It follows from this that every homology class of degree $2$ in $\X^\wt0$ 
is represented by a cycle of the form 
  \(
  r\sigma(\tfrac{a}{c})|\YX - y r \sigma(\tfrac{a'}{c})|\HX
  \)
with $r\in k[h]$ with $\deg r<\deg c=M$. In view of the formula
\eqref{eq:delta1}, one of these cycles is a boundary if and only if it is
zero, and we can then conclude that  $H_2(\X^\wt0)\cong
k[h]/(\sigma(c))\cong k[h]/(c)$.

\item It follows immediately from \eqref{eq:delta2} that
$\delta(\X_2^\wt0)=\sigma(c)k[h]|\YHX$, so $H_3(\X^\wt0)\cong k[h]/(c)$.

\end{itemize}

\bigskip

We fix now $r>0$, and show that $\X^\wt r$ is exact.

\begin{itemize}
\item
Let $u\in \X^\wt r_0$, so that $u=y^r p$ for some $p\in k[h]$.
Then
  \begin{equation} \label{eq:deltaX-0r}
  \delta(u) 
        = y^{r-1}\sigma^{r}(a)p|Y 
        - y^r p\tsum_i \alpha_i \q{i}{q^r} h^{i-1}|H
        + y^{r+1}\sigma(p)|X,
  \end{equation}
and we see immediately that this is zero if and only if $p=0$, so
$H_0(\X^\wt r)=0$.

\item Let $u=y^{r-1}p_1|Y + y^r p_2 |H + y^{r+1}p_3 |X\in \X^\wt r_1$ with
$p_1$, $p_2$, $p_3\in k[h]$. As
  \begin{multline*}
  \delta(u) 
        = y^{r-1}\bigl(p_1 \tsum_i \alpha_i \q{i}{q^r} h^{i-1}+\sigma^r(a)p_2\bigr)|\YH  
        + y^r(-\sigma(p_1)+\sigma^{r+1}(a)p_3)|\YX \\
        - y^{r+1}\bigl(\sigma(p_2)+p_3 \tsum_i \alpha_i q^{i-1}\q{i}{q^r} h^{i-1}\bigr)|\HX,
  \end{multline*}
we have that $u$ is a cycle if and only if
  \begin{gather*}
   p_1 \tsum_i \alpha_i \q{i}{q^r} h^{i-1}+\sigma^r(a)p_2=0, \\
   \sigma^{r+1}(a)p_3=\sigma(p_1), \\
   \sigma(p_2)+p_3 \tsum_i \alpha_i q^{i-1}\q{i}{q^r} h^{i-1}=0.
  \end{gather*}
The first one follows from the other two, so we can drop it, and we can
replace the remaining ones by
  \begin{gather*}
   p_2=-\sigma^{-1}(p_3) \tsum_i \alpha_i \q{i}{q^r} h^{i-1}, \\
   p_1=\sigma^r(a)\sigma^{-1}(p_3).
  \end{gather*}
We thus obtain a description of all $1$-cycles in $\X^\wt r$ and
comparing it with \eqref{eq:deltaX-0r}, we see
that they are all boundaries: it follows that $H_1(\X^\wt r)=0$.

\item
For $u=y^{r-1}p_1|\YH + y^r p_2 |\YX + y^{r+1}p_3 |\HX\in\X^\wt r_2$ with
$p_1$, $p_2$, $p_3\in k[h]$, we have
  \begin{equation*} 
  \delta(u)
        = y^r\bigl(
                \sigma(p_1) 
                + p_2\tsum_i\alpha_iq^{i-1} \q{i}{q^r} h^{i-1}
                +\sigma^{r+1}(a)p_3
             \bigr)|\YHX.
  \end{equation*}
If $u$ is a cycle, then 
  \(
   p_1=-\sigma^{-1}(p_2\sum_i\alpha_iq^{i-1}\q{i}{q^r}h^{i-1}+\sigma^{r+1}(a)p_3)
  \)
so that, in fact, 
  \[
  u=-\delta(y^{r-1}\sigma^{-1}(p_2)|Y+y^r \sigma^{-1}(p_3) |H).
  \]
It follows from this that $H_2(\X^\wt r)=0$.

\item For each $p\in k[h]$, we have that
$\delta(y^{r-1}\sigma^{-1}(p)|\YH)=y^rp|\YHX$. This means that
$\delta(\X_2^r)=\X_3^r$, so $H_3(\X^\wt r)=0$. \qedhere

\end{itemize}
\end{proof}

At this point, we know most of the second page of our spectral sequence:

\begin{Corollary}\label{coro:e1}
Let $r\in\ZZ$ a weight. The dimensions of the vector spaces appearing in
the homogeneous component of weight $r$ of $E^1$ are
  \[
   \vcenter{\xymatrix@R=2ex@C=2ex{
   {} 
        &   
        &   
        & M 
        & ? 
        & ? 
        & ? 
        \\
  {}    
        &   
        & M 
        & M    
        & 0    
        & 0    
        \\ 
   {} 
        & M 
        & M    
        & 0 
        & 0 
        \\ 
   \iddots    
        & \iddots    
        & \iddots 
        & \iddots 
        \\ 
  }}
  \qquad\text{or}\qquad
  \vcenter{\xymatrix@R=2ex@C=2ex{
  {} 
	& 
	& 
	& 0 
	& ? 
	& ? 
	& ? 
        \\
  {}
	& 
	& 0 
	& 0 
	& 0 
	& 0 
	\\ 
  {}
	& 0 
	& 0 
	& 0 
	& 0 
	\\ 
  \iddots 
	& \iddots
	& \iddots
	& \iddots
	\\
  }}
  \]
depending on whether $r=0$ or not. The question marks denote vector spaces
for which we still do not know the dimension.~\qed
\end{Corollary}

\subsection{Second page}

In view of the shape of $E^1$, we have $E^\infty=E^2$. The following
proposition takes care of the latter, except for its first row, and the
rest of this section will be devoted to the computation of the few
remaining vector spaces.

\begin{Proposition}
For each $p\geq0$, the differential $d^1_{p+3,p}:E^1_{p+3,p}\to
E^1_{p+2,p}$ vanishes. In consequence, except for the vector spaces denoted
with question marks in the diagrams of Corollary~\ref{coro:e1}, the
$E^\infty$ page coincides with the page $E^1$.
\end{Proposition}

\begin{proof}
A simple computation shows that if $f\in k[h]$ then
  \begin{equation}\label{eq:e2-1}
  \begin{aligned} 
   d(f|\YHX)
	&= y(1-q\sigma)(f)|\HX-(1-q\sigma)(f)x|\YH \\
	&= \delta((q-\sigma^{-1})(f)|H).
  \end{aligned}
  \end{equation}
It follows that all the differentials $d^2_{p+3,p}$ are zero, as claimed, and
the computation of $E^\infty$ is immediate except for $E^\infty_{0,0}$,
$E^\infty_{1,0}$ and $E^\infty_{2,0}$.
\end{proof}

\begin{Corollary}
For all $p\geq3$ and all $r\in\ZZ$ there are isomorphisms of $S$-modules
  \[
  HH_p(A)_{(r)} \cong 
                \begin{cases}
  		    k[h]/(c) & \text{if $r=0$;} \\
                    0   & \text{if $r\neq 0$.} 
                \end{cases}
  \]
\end{Corollary}

\noindent Notice that this result is independent of $q$.

\begin{proof}
According to the proposition and in view of the shape of the $E^1$ page of
the spectral sequence, this is a consequence of convergence.
\end{proof}

To finish the computation, we need to take care of the spots in the
spectral sequence tagged with question marks in the diagrams of
Corollary~\ref{coro:e1}. We do this in the following two propositions,
first for weight zero and then for the remaining ones.

\begin{Proposition}
When $q$ is a root of unity, we have isomorphisms of $\S$-modules
  \begin{gather*}
  HH_p(A)^\wt0 \cong E^{2\wt0}_{p,0} \cong 
        \begin{cases}
        k^{\eta(a)}, & \text{if $p=0$;}  \\
        \S \oplus \S \oplus k^{\eta(c)}, & \text{if $p=1$;} \\
        \S \oplus k[h]/(c),& \text{if $p=2$;} 
        \end{cases}
  \end{gather*}
with $\eta(f)=N-\tfrac1e\deg\N(f)$ for $f\in k[h]$ as in
Lemma~\ref{lema:coker}. On the other hand, if $q$ is of infinite order we
have isomorphisms
  \begin{gather*}
  HH_p(A)^\wt0 \cong E^{2\wt0}_{p,0} \cong 
        \begin{cases}
        k^N, & \text{if $p=0$;} \\
        k^M, & \text{if $p=1$;} \\
        k^M, & \text{if $p=2$.}
        \end{cases}
  \end{gather*}
\end{Proposition}

\begin{proof} 
We write $E^1_{p,0}$ instead of $E^{1\wt0}_{p,0}$ throughout this proof, to
lighten the notation.

\begin{itemize}[fullwidth, font=\bfseries]

\item[Homology at \texorpdfstring{$E^1_{2,0}$}{E^1_20}.]
Suppose $u=p_1x|\YH+p_2|\YX+yp_3|\HX \in E^0_{2,0}$, with $p_1$, $p_2$,
$p_3\in k[h]$, lives to $E^2$, so that there exists an $f \in k[h]$ such
that $d(u)=\delta(f)$. This means that
  \begin{gather*}
  (1-\sigma)(p_2)=\sigma(f), \\
  a\sigma^{-1}(p_1+p_3)-q\sigma(a)(p_1+p_3)-p_2(q\sigma(a')-a') =-a'f.
  \end{gather*}
Since $\sigma$ is a automorphism, we can eliminate $f$ obtaining the
equivalent equation
  \[
  a\sigma^{-1}(p_1+p_3) - q\sigma(a)(p_1+p_3) - p_2(q\sigma(a')-a') 
        = -a'\sigma^{-1}((1-\sigma)(p_2)),
  \]
which we can rewrite more compactly as
  \begin{equation}
  \label{eq:deg-2}
  (1-q\sigma)(a\sigma^{-1}(p_1+p_3)+a'\sigma^{-1}(p_2))=0.
  \end{equation}
It will be necessary to treat two cases separately, since the result
depends on whether $q$ is a root of unity or not. 

\begin{itemize}[label=\textbullet]

\item Suppose first that \emph{$q$ is not a root of $1$}. In this case, the
map $1-q\sigma$ is a monomorphism, so \eqref{eq:deg-2} is the same as
  \[
   a\sigma^{-1}(p_1+p_3)+a'\sigma^{-1}(p_2)=0.
  \]
From this it follows that there exists $g\in k[h]$ such that
  \begin{align*}
  & p_2=-\sigma(\tfrac{a}{c})g,
  && p_1+p_3=\sigma(\tfrac{a'}{c})g.
  \end{align*}
Let $b$, $r\in k[h]$ be such that $g=b\sigma(c)+r$ with $\deg r<\deg c$.
Then $u$ is homologous to 
  \[
  u+\delta\bigl(yb|X-\sigma^{-1}(p_1)|H\bigr) 
        = \sigma(\tfrac{a}{c})r|\YX+y\sigma(\tfrac{a'}{c})r|\HX,
  \]
and we see that every homology class in $E^2_{2,0}$ is represented by a
cycle of the form 
  \begin{equation} \label{eq:2-co}
  \sigma(\tfrac{a}{c})r|\YX+y\sigma(\tfrac{a'}{c})r|\HX
  \end{equation}
with $r\in k[h]$ with $\deg r<M=\deg c$. Conversely, each element of this
form lives to $E^2$. 

Using \eqref{eq:e2-1} we see that the image of $d$ contains the image of
$\delta$. On the other hand, the coefficient of $\YX$ in every non zero
element of $\delta(\X_1^\wt0)$ is multiple of $\sigma(a)$, so in particular
it has degree at least $N$: comparing with~\eqref{eq:2-co} we see that $u$
is not in the image of~$\delta$. We can therefore conclude that these
elements are non zero in~$E^2$, so that $\dim E^2_{2,0}=M$.

\item Suppose now that \emph{$q$ is a root of $1$}. In this case the
condition \eqref{eq:deg-2} is equivalent to the existence of a singular
polynomial $s\in\S$ such that 
  \begin{equation} \label{eq:aaa}
   a\sigma^{-1}(p_1+p_3)+a'\sigma^{-1}(p_2)=h^{e-1}s.
  \end{equation}
As $a(0)\neq0$, $c$ divides $s$ and it follows from
Proposition~\ref{prop:barra}\emph{(ii)} that there exists $s_1\in\S$ such
that $s=\N(c)s_1$.

Let $\alpha$, $\beta\in k[h]$ be such that
$\tfrac{a}{c}\alpha+\tfrac{a'}{c}\beta=1$; each solution of the
equation~\eqref{eq:aaa} is of the form
  \begin{gather*}
  p_3=\sigma\bigl(h^{e-1}\overline{c}s_1\alpha+\tfrac{a'}{c}g\bigr)-p_1, \\
  p_2=\sigma\bigl(h^{e-1}\overline{c}s_1\beta-\tfrac{a}{c}g\bigr)
  \end{gather*}
for some $g\in k[h]$. Let $b$, $r\in k[h]$ be such $g=bc+r$ and $\deg r<M$.
Without changing its class in $E^2$, we can replace $u$ by
$u-\delta(\sigma^{-1}(p_1)|H-y\sigma(b)|X)$, and then we see that we may
assume that 
  \begin{equation} \label{eq:e2-g2}
   u = \sigma(h^{e-1}\overline{c}s_1\beta-\tfrac{a}{c}r)|\YX
     + y\sigma(h^{e-1}\overline{c}s_1\alpha+\tfrac{a'}{c}r)|\HX.
  \end{equation}

If $u$ represents the zero class in $E^1$, then there exist
$v_1$, $v_2$, $v_3\in k[h]$ such that 
  \begin{align*}
  u &= \delta(v_1x|Y+v_2|H+yv_3|X) \\
    &= (v_1\sigma(a')+\sigma(v_2))x|\YH
        + \sigma(a)(v_3-v_1)|\YX
        - y(v_3\sigma(a' )+\sigma(v_2)|\HX.
  \end{align*}
Equating coefficients and eliminating $v_2$, we see that
  \begin{gather*}
  a\sigma^{-1}(v_3-v_1) = h^{e-1}\Bar cs_1\beta-\tfrac acr, \\
  -a'\sigma^{-1}(v_3-v_1) = h^{e-1}\Bar cs_1\alpha+\tfrac {a'}cr.
  \end{gather*}
Solving now for $s_1$ and then for $r$, we see that $u$ must be zero.

Let us show now $u$ represents a non zero element of $E^2$. Indeed, if
there exists a $p\in k[h]$ such that
  \[
  u = d(p|\YHX) = y(1-q\sigma)(p)|\HX - (1-q\sigma)(p)x|\YH,
  \]
then we must have $(1-q\sigma)(p)=0$ and 
  \begin{gather*}
  \tfrac{a}{c}r=h^{e-1}\overline{c}s_1\beta, \\
  \tfrac{a'}{c}r=-h^{e-1}\overline{c}s_1\alpha.
  \end{gather*}
Solving these equations for $s_1$ and $r$, recalling the way $\alpha$ and
$\beta$ were chosen, and using that $h^{e-1}\Bar c\neq0$, we see that
$s_1=r=0$.

We conclude in this way that every element of $E^2_{2,0}$ is represented
uniquely by a cycle of the form \eqref{eq:e2-g2}. In particular, we have a
vector space isomorphism $E^2_{2,0}\cong\S\oplus k[h]/(c)$.

\end{itemize}

\item[Homology at \texorpdfstring{$E^1_{1,0}$}{E^1_10}.] Let
$u=p_1x|Y+p_2|H+yp_3|X\in E^0_{1,0}$, with $p_1$, $p_2$, $p_3\in k[h]$, an
element which survives to $E^2$. As $u$ is homologous to
  \(
  u-\delta(\sigma^{-1}(p_1)) = (p_2+a'\sigma^{-1}(p_1))|H + y(p_3-p_1)|X,
  \)
we can suppose that $p_1=0$. 

If $u$ is a boundary, so that $u=d(f_1x|\YH+f_2|\YX+yf_3|\HX)+\delta(f_4)$,
for some $f_i\in k[h]$, looking at the coefficient of $Y$ on both sides of
this equality we find that $(1-\sigma)(f_2)+\sigma(p_4)=0$. This implies
that $p_3=0$ and that $p_2\in(1-q\sigma)((c))$. On the other hand, since
  \begin{equation} \label{eq:du-3}
 d(u)=\sigma(a)p_3-a\sigma^{-1}(p_3)=(\sigma-1)(a\sigma^{-1}(p_3)) = 0,
  \end{equation}
we see that $a\sigma^{-1}(p_3)\in\S$.

\begin{itemize}[label=\textbullet]

\item Suppose first that \emph{$q$ is a root of $1$}.  Then
$p_3=\sigma(\overline{a})s$ for some $s\in\S$, according to
Proposition~\ref{prop:barra}, and thus we have
$u=p_2|H+y\sigma(\overline{a})s|X$. In view of the description given above
for the boundaries, we conclude that
  \[
  E^2_{1,0} \cong 
  	\frac{k[h]}{(1-q\sigma)((c))}|H  	
	\oplus y\sigma(\bar a)\S|X.
  \]
Using Lemma~\ref{lema:coker} we see that the first summand is isomorphic to
$k^{\eta(c)}\oplus S$.

\item Suppose next that \emph{$q$ is not a root of $1$}.  In this case,
since $a$ is not constant, equation~\eqref{eq:du-3} implies that $p_3=0$.
Using again the description of boundaries, we have 
  \[
  E^2_{1,0} \cong \frac{k[h]}{(1-q\sigma)((c))}|H,
  \]
a vector space of dimension $M$.

\end{itemize}

\item[Homology at \texorpdfstring{$E^1_{0,0}$}{E^1_00}.]
We have to compute the cokernel of the map $d:A|V\to A$. One sees at once
that its image coincides with the image of the map $\psi_{a,0}:f\in
k[h]\mapsto (\sigma-1)(af)\in k[h]$ from Lemma~\ref{lema:coker}. If $q$ is
not a root of unity, it is immediate that the classes of $1$, \dots,
$h^{N-1}$ freely span $\coker\psi_{a,0}$, so that $\dim E^2_{0,0}=N$. On
the other hand, if $q$ is a root of unity, then Lemma~\ref{lema:coker}
tells us that the dimension of the cokernel of~$\psi_{a,0}$, equal to
that of $E^2_{0,0}$, is $\eta(a)=N-\frac1e\deg\N(a)$. \qedhere
\end{itemize}

\end{proof}

\begin{Proposition}
Let $r\neq0$. According to whether $r$ is regular or not, 
there are isomorphisms of $\S$-modules
  \begin{gather*}
  HH_p(A)^\wt r \cong E^{2\wt r}_{p,0} \cong
                \begin{cases}
                \S, & \text{if $p=0$;} \\
                k, & \text{if $p=1$;}  \\
                0, & \text{if $p=2$.}
                \end{cases}
    \\
\shortintertext{or}
  HH_p(A)^\wt r \cong E^{2\wt r}_{p,0} \cong
                \begin{cases}
                \S, & \text{if $p=0$;} \\
                \S\oplus\S, & \text{if $p=1$;}  \\
                \S, & \text{if $p=2$.}
                \end{cases}
  \end{gather*}
\end{Proposition}

\begin{proof}
By symmetry, we can consider just the case where $r>0$.

\begin{itemize}[fullwidth, font=\bfseries]

\item[Homology at $E^{1\wt r}_{2,0}$.]
Let $u\in E^{0\wt r}_{2,0}$ be an element representing a cycle in
$E^1$. It follows that $u=y^{r-1}p_1|\YH + y^r p_2 |\YX + y^{r+1}p_3 |\HX$ with
$p_1$, $p_2$, $p_3\in k[h]$. Without loss of generality, we can assume that
$p_2=p_3=0$; if that is not the case, we can replace $u$ by
  \[
 u+\delta\Bigl(y^{r-1}\sigma^{-1}(p_2)|Y + y^r\sigma^{-1}(p_3)|H\Bigr)
  \]
without changing the class of $u$ in $E^1$.
Computing, we find then that
  \begin{equation}\label{eq:delta-u}
   d(u) = y^{r-1}(1-q^r)p_1h|Y +y^r(p_1-q\sigma(p_1))|H.
  \end{equation}
Comparing with equation~\eqref{eq:delta-pri} we see that, since $d(u)$ is
in the image of~$\delta$, $(1-q^r)p_1=0$ and $p_1-q\sigma(p_1)=0$. If $r$
is a regular weight, it follows that $p_1=0$, so $E^{2\wt r}_{2,0}=0$. On
the other hand, if $r$ is singular, these equations are satisfied if and
only if $p_1\in h^{e-1}\S$: in this case we have $E^{2\wt r}_{2,0}\cong
h^{e-1}\S$.

\item[Homology at $E^{1\wt r}_{1,0}$.] Let $u=y^{r-1}p_1|Y + y^r p_2 |H +
y^{r+1}p_3 |X\in E^{0\wt r}_{1,0}$, with $p_1$, $p_2$, $p_3\in k[h]$, an
element which lives to $E^2$. Up to replacing $u$ by
$u-\delta(y^r\sigma^{-1}(p_3))$, we can assume that $p_3=0$, so that 
  \begin{equation} \label{eq:bbb}
  d(u)=y^r(p_1 - \sigma(p_1) + (q^r-1)h p_2) = 0.
  \end{equation}

Assume that $r$ is singular. It follows that $p_1\in\S$; moreover, in view
of  formula~\eqref{eq:delta-u}, we can reduce $p_2$ modulo the image of
$1-q\sigma$, so that we can suppose that $p_2\in h^{e-1}\S$. From equations
\eqref{eq:d1-yh}, \eqref{eq:d1-yx}, \eqref{eq:d1-hx} and
\eqref{eq:delta-pri} we see then that $u$ is not a boundary and we conclude
that $E^{2\wt r}_{1,0}\cong\S\oplus\S$ in this case, freely generated as a
$\S$-module by the classes of $y^{r-1}|Y$ and $y^r h^{e-1}|H$.

Finally, let us assume that $r$ is regular. Using again \eqref{eq:delta-u},
we see that we can now replace $u$ by an homologous element of the same
form but now with $p_1\in k$ and then, because of~\eqref{eq:bbb}, we must
have $p_2=0$. In this way, we see that $u$ must be a scalar multiple of
$y^{r-1}|Y$. If such an element is a boundary, looking at the constant term
in the formulas \eqref{eq:d1-yh}, \eqref{eq:d1-yx}, \eqref{eq:d1-hx} and
\eqref{eq:delta-pri}, we infer that $u$ is zero, therefore $E^{2\wt
r}_{1,0}$ is one-dimensional. 

\item[Homology at \texorpdfstring{$E^{1\wt r}_{0,0}$}{E^1r_00}.] Let $u=y^r
p\in E^{0\wt r}_{0,0}$. We can add to $u$ elements in the image of $d$
without changing its homology class; doing so, we can assume that $p\in S$.
Moreover, $u$ itself is then not in the image of $d$: this means that
$E^{1\wt r}_{0,0}\cong S$, freely generated by the class of $1$.~\qedhere

\end{itemize}

\end{proof}

\section{Hochschild cohomology}
\label{cohomology}

In this section we compute the Hochschild cohomology of $A$ using, as
before, a spectral sequence. Write $\hV=\hom_k(V,k)$, and let
$\{\hY,\hH,\hX\}$ be the basis of $\hV$ dual to $\{Y,H,X\}$. We identify in
the usual way $\hom_k(\bigwedge^p V,k)$ with $\bigwedge^p \hV$. Applying
the functor $\hom_{A^e}(-,A)$ to the resolution constructed in
\ref{resolution} we obtain a double complex whose cohomology is the
Hochschild cohomology $HH^{\bullet}(A)$ of~$A$. After we identify
$\hom_{A^e}(A|\bigwedge^p V |A,A)$ with $A|\bigwedge^p \hV$ in the natural
way, this double complex is
  \[
  \xymatrix{
  {} 
	& 
	& 0 \ar[d]^-{\delta} 
	& A|\bigwedge^3 \hV \ar[l]_d \ar[d]^-{\delta} 
	& A|\bigwedge^2 \hV \ar[l]_d \ar[d]^-{\delta} 
	& A| \hV \ar[l]_d \ar[d]^-{\delta} 
	& A \ar[l]_d 
	& 
        \\
  {}
	& 0 \ar[d]^-{\delta} 
	& A|\bigwedge^3 \hV \ar[l]_d \ar[d]^-{\delta} 
	& A|\bigwedge^2 \hV \ar[l]_d \ar[d]^-{\delta} 
	& A| \hV \ar[l]_d \ar[d]^-{\delta} 
	& A \ar[l]_d
	\\
   \iddots 
	& \iddots 
	& \iddots 
	& \iddots 
	& \iddots 
  }
  \]
with differentials given by
  \begin{subequations}
  \begin{gather}
  d(u) = [u,y]|\hY + [u,h]|\hH + [u,x]|\hX; \label{eq:cd-0} \\
  d(u|\hY) = [h,u]_q|\hYH - [u,x]|\hYX; \label{eq:cd-y} \\
  d(u|\hH) = [x,u]_q|\hHX - \tsum_i \lambda_i \tint_ih^s u h^t |\hYX + [u,y]_q|\hYH; \label{eq:cd-h} \\
  d(u|\hX) = [u,h]_q|\hHX + [u,y]|\hYX; \label{eq:cd-x} \\
  d(u|\hYH) = -[x,u]_q|\hYHX; \label{eq:cd-yh} \\
  d(u|\hYX) = q[h,u]|\hYHX; \label{eq:cd-yx} \\
  d(u|\hHX) = [u,y]_q|\hYHX; \label{eq:cd-hx}
  \end{gather}
  \end{subequations}
and
  \begin{subequations}
  \begin{gather}
  \delta(u|\hY) = ux; \label{eq:cdelta-y} \\
  \delta(u|\hH) = -\tsum_i\alpha_i \tint_ih^s u h^t; \label{eq:cdelta-h} \\
  \delta(u|\hX) = yu; \label{eq:cdelta-x} \\
  \delta(u|\hYH) = \tsum_i\alpha_i\tint_iq^th^suh^t|\hY+ux|\hH; \label{eq:cdelta-yh} \\
  \delta(u|\hYX) = ux|\hX - yu|\hY; \label{eq:cdelta-yx} \\
  \delta(u|\hHX) = -yu|\hH-\tsum_i\alpha_i\tint_iq^sh^suh^t|\hX; \label{eq:cdelta-hx} \\
  \delta(u|\hYHX) = ux|\hHX+\tsum_i\alpha_i\tint_iq^{i-1}h^suh^t|\hYX+yu|\hYH; \label{eq:cdelta-yhx}
  \end{gather}
  \end{subequations}
We consider the spectral sequence $E$ which arises from the filtration of
this double complex by columns.

\subsection{First Page}

In this section we deal with the first page of the spectral sequence. Let
$\Y$ be the complex
  \begin{equation}
  \label{col}
  \xymatrix{
    0 \ar[r] 
    & A|\bigwedge^3 \hV \ar[r]^-{\delta} 
    & A|\bigwedge^2 \hV \ar[r]^-{\delta} 
    & A| \hV \ar[r]^-{\delta} 
    & A 
  }
  \end{equation}
with differentials as in \eqref{eq:cdelta-y}--\eqref{eq:cdelta-yhx}. As
before, we have $E_1^{p,q}\cong H^{p-q}(\Y)$ for all $q>0$ and the vector
spaces $E_1^{p,0}$ are isomorphic to the kernels of the differentials
of~$\Y$. For each $r\in\ZZ$ we denote $\Y_\wt r$ the component of weight
$r$ in $\Y$, and extend this notation to related objects.

\begin{Proposition}\label{prop:ce1}
If $r\in\ZZ$ is non zero, then the complex $\Y_\wt r$ is exact. On the
other hand, there are $\S$-module isomorphisms
  \[
  H^p(\Y_\wt0) \cong 
    \begin{cases}
    k[h]/(c), & \text{if $0\leq p\leq 1$;} \\
    0, & \text{otherwise.}
    \end{cases}
  \]
\end{Proposition}

\begin{proof}
We prove this by computing the relevant homology groups:
\begin{itemize}[font=\bfseries]

\item If $u=p|\hYHX \in\Y^3_\wt0$ with $p\in k[h]$, then
  \begin{equation} \label{eq:y-3}
  \delta(u)=px|\hHX+p \sigma(a')|\hYX+yp|\hYH.
  \end{equation}
It is clear then that $H^3(\Y_\wt0)=0$.

\item Let $u=y p_1|\hYH + p_2 |\hYX + p_3 x|\hHX \in \Y^2_\wt0$ with $p_1$,
$p_2$, $p_3\in k[h]$. One can see that
  \begin{equation} \label{eq:y-1}
  \delta(u) = y(p_1 \sigma(a')-p_2)|\hY 
            + (a\sigma^{-1}(p_1-p_3))|\hH 
            + (-p_3\sigma(a') + p_2)x|\hX.
  \end{equation}
In particular, if $u$ is a cycle, $p_2=\sigma(a')p_1$ and $p_3=p_1$.
Comparing with the expression~\eqref{eq:y-3} for $2$-boundaries in $\Y$, we
see at once that $H^2(\Y_\wt0)=0$.

\item Finally, let $u=y p_1|\hY + p_2 |\hH + p_3 x|\hX \in \Y^1_\wt0$, with
$p_1$, $p_2$, $p_3\in k[h]$, a $1$-cycle. Since we can replace $u$ for
$u+\delta(p_1|\hH)$, without changing the homology class it represents, we
can assume that $p_1=0$, and then $\delta(u)=a\sigma^{-1}(p_3)-p_2a' = 0$.
It follows that there exists $g\in k[h]$ such that
$p_3=\sigma(\tfrac{a'}{c}g)$ and $p_2=\tfrac{a}{c}g$. Let $b$, $r\in k[h]$
such that $g=bc+r$ and $\deg r<M$. Then
  \[
  u + \delta(\sigma(b)x|\hHX) = \tfrac{a}{c}r |\hH +  \sigma(\tfrac{a'}{c} r)x|\hX
  \]
This means that all classes in $H^1(\Y_\wt0)$ can be represented by a
element of the form $\tfrac{a}{c}r |\hH +  \sigma(\tfrac{a'}{c} r)x|\hX$
with $r\in k[h]$ and $\deg r<M$ and, moreover, such an element represents
the zero class only when it is itself zero: this can be seen by looking at
the degree of the coefficient of $\hH$ appearing the formula~\eqref{eq:y-1}
for $1$-boundaries. Conversely, every such element is a cycle. We conclude
that $H^1(\Y_\wt0)\cong k[h]/(c)$.

\item If $u=y p_1|\hY + p_2 |\hH + p_3 x|\hX \in \Y^1_\wt0$, with $p_1$,
$p_2$, $p_3\in k[h]$, then $\delta(u)=a\sigma^{-1}(p_1+p_3)-p_2a'$, so
$H^0(\Y_\wt0)\cong k[h]/(c)$.

\end{itemize}

It remains to check, in these last two items, that the obtained
isomorphisms are $\S$-linear: this is just a matter of following the
computation, and we omit the details.

Let us now fix $r>0$.

\begin{itemize}[resume]

\item If $u=y^rp|\hYHX\in\Y_\wt r^3$, with $p\in k[h]$, then
  \begin{equation}\label{eq:cd-3}
   \delta(u) =
        y^{r+1} p |\hYH
        + y^r p \tsum_i \alpha_i q^{i-1} \q{i}{q^r}h^{i-1}|\hYX 
        + y^{r-1}a\sigma^{-1}(p)|\hHX.
  \end{equation}
Looking at the coefficient of $\hYH$ we see that $u$ is a cycle if and only
if $u$ is zero, so $H^3(\Y_\wt r)=0$.

\item Let $u=y^{r+1}p_1|\hYH + y^r p_2 |\hYX + y^{r-1}p_3|\hHX\in\Y^2_\wt
r$, with $p_1$, $p_2$, $p_3\in k[h]$. Since
  \begin{multline*}
   \delta(u) =
        y^{r+1}(p_1\tsum_i\alpha_iq^{i-1}\q{i}{q^r}h^{i-1}-p_2)|\hY 
        + y^r( a \sigma^{-1}(p_1)-p_3)|\hH \\
        + y^{r-1}(a\sigma^{-1}(p_2)-p_3\tsum_i\alpha_i\q{i}{q^r}h^{i-1})|\hX,
  \end{multline*}
it is  easy to see that $u$ is a cycle if and only if
  \begin{gather*}
  p_2=p_1\sum_i\alpha_iq^{i-1}\q{i}{q^r}h^{i-1}, \\
  p_3=a \sigma^{-1}(p_1),
  \end{gather*}
and in that case, according to~\eqref{eq:cd-3}, we have
$u=\delta(y^rp_1|\hYHX)$. We conclude that $H^2(\Y_\wt r)=0$.

\item Let $u=y^{r+1}p_1|\hY +y^r p_2 |\hH + y^{r-1}p_3|\hX\in\Y^1_\wt r$,
with $p_1$, $p_2$, $p_3\in k[h]$ a cycle. Without changing its homology
class, we can replace $u$ by $u+\delta(y^r p_1 |\hYX + y^{r-1}p_2|\hHX)$,
and hence we can suppose that $p_1=p_2=0$. In that case $\delta(u)=y^rp_3$,
and we see that $u=0$. It follows that $H^1(\Y_\wt r)=0$.

\item Finally, for each $p\in k[h]$, $\delta(y^{r-1}p|\hX)=y^rp$, so that
$\delta(\Y^1_\wt r)=\Y^0_\wt r$ and $H^0(\Y_\wt r)=0$. \qedhere

\end{itemize}
\end{proof}

\begin{Corollary}\label{coro:ce1}
If $r\in\ZZ$, the dimensions of the vector spaces appearing in the
component $E_{1\wt r}$ of $E_1$ are
  \[
  \vcenter{\xymatrix@R=2ex@C=2ex{
       &   &   & 0 & ? & ? & ? \\
       &   & 0 & 0 & M & M     \\ 
       & 0 & 0 & M & M         \\ 
     \iddots & \iddots & \iddots & \iddots \\ 
    }}
  \qquad\text{or}\qquad
  \vcenter{\xymatrix@R=2ex@C=2ex{
       &   &   & 0 & ? & ? & ? \\
       &   & 0 & 0 & 0 & 0     \\
       & 0 & 0 & 0 & 0         \\
     \iddots & \iddots & \iddots & \iddots \\ 
    }}
  \]
if $r=0$ or $r\neq0$, respectively. The question marks denote vector spaces
for which we still do not know the dimension.
\end{Corollary}

\begin{proof}
This follows from the proposition and the isomorphisms $E_{1\wt
r}^{p,q}\cong H^{p-q}(\Y_\wt r)$.
\end{proof}

\subsection{The second page}

\begin{Proposition}
For each $p\geq0$, the differential $d_1^{p,p}:E_1^{p,p}\to E_1^{p+1,p}$
vanishes. The page $E_\infty$ then coincides with $E_1$, except at the
places marked with question marks in the diagrams of
Corollary~\ref{coro:ce1}, and we have
  \[
  HH^p(A)_\wt r \cong
    \begin{cases}
    k[h]/(c), & \text{if $r=0$;} \\
    0, & \text{if $r\neq0$.}
    \end{cases}
  \]
\end{Proposition}

\begin{proof}
The set of homology classes of the elements of $\{h^l:0\leq l<M\}$ is a
basis of the space $E_1^{p,p}$, and
  \[
  d(h^l)=(q^l-1)yh^l|\hY - (q^l-1)h^lx|\hX=\delta(-(q^l-1)h^l |\hYX).
  \]
It follows that $d_1^{p,p}$ is indeed zero, as claimed. The rest of the
proposition is then a consequence of the fact that the spectral sequence
$E$ converges to $HH^\bullet(A)$.
\end{proof}

\begin{Proposition}
\label{prop:co-e2-0}
If $q$ is a root of unity, then
  \begin{gather*}
  E_{2\wt0}^{p,0} \cong 
        \begin{cases}
        \S, & \text{if $p=0$;}  \\
        \S \oplus \S, & \text{if $p=1$;} \\
        \S \oplus k^{\eta(a/c)}, & \text{if $p=2$,}
        \end{cases}
  \end{gather*}
where, as in Lemma~\ref{lema:coker}, $\eta(a/c)=N-M-\deg\N(a/c)/e$, and if
$q$ has infinite order,
  \begin{gather*}
  E_{2\wt0}^{p,0} \cong 
        \begin{cases}
        k, & \text{if $p=0$;}  \\
        k, & \text{if $p=1$;} \\
        k^{N-M}, & \text{if $p=2$.}
        \end{cases}
  \end{gather*}
\end{Proposition}

\begin{proof}
We write, during this proof,  $E_{r}^{p,q}$instead of $E_{r\wt 0}^{p,q}$
for simplicity.
\begin{itemize}[fullwidth, font=\bfseries]

\item[Homology at $E_1^{0,0}$.]
If $u=p\in E_1^{0,0}$, so that in fact $p\in k[h]$, we have
  \begin{equation} \label{eq:du0}
  d(p)=y(\sigma(p)-p)|\hY-(\sigma(p)-p)x|\hX.
  \end{equation}
It follows that $E_1^{0,0}=\ker(\sigma-1)=\S$.

\item[Homology at $E_1^{1,0}$.]
If $u\in E_1^{1,0}$, there exist $p_1$, $p_2\in k[h]$ such that $u=y
p_1|\hY + \tfrac{a}{d}p_2 |\hH +  (\sigma(\tfrac{a'}{d} p_2)-p_1)x|\hX$;
this is a consequence of the formulas \eqref{eq:cdelta-y},
\eqref{eq:cdelta-h} and \eqref{eq:cdelta-x} using the same reasoning as in
the third step of the proof of Proposition~\ref{prop:ce1}. Moreover, there
exist $s_1\in\S$ and $b\in k[h]$ such that $p_1=s_1+(\sigma-1)(b)$ and we
can replace $u$ by $u-d(b)$ so, in the end, we can assume that
$p_1=s_1\in\S$. In that case, $u$ is boundary only if it zero: this follows
by comparing with the coefficient of~$\hY$ in~\eqref{eq:du0}. Computing, we
find that
  \begin{multline*}
  d(u) = (\sigma-q)(\tfrac{a}{c}p_2)x|\hHX + y(\sigma-q)(\tfrac{a}{c}p_2)|\hYH \\
     +\bigl((\sigma-1)(\tfrac{aa'}{c}p_2)- \tfrac{a}{c}p_2(q\sigma(a')-a')\bigr)|\hYX.
  \end{multline*}
If $d(u)=0$, then $(\sigma-q)(\tfrac{a}{c}p_2)=0$ and $\frac{a}{c}p_2\in
h\S$; conversely, if $\frac{a}{c}p_2\in h\S$, then $u$ is a cycle. we treat
separately two cases, according to whether $q$ is a root of unity or not.

\begin{itemize}[label=\textbullet]

\item Suppose first that \emph{$q$ is not a root of $1$}. As
$\tfrac{a}{c}p_2\in h\S$ and $\S=k$, then $p_2\in k$. Evaluating
$\tfrac{a}{c}p_2$ at zero, and using the hypothesis that $a(0)\neq0$, we
see that $p_2=0$. In this case, then, $u$ is a scalar multiple of
$y|\hY-x|\hX$. Since all such non zero multiples are cycles and not
boundaries, we conclude that $E_2^{1,0}$ is one dimensional, generated by
the class of $y|\hY-x|\hX$.

\item Suppose now that \emph{$q$ is a root of $1$}. As $h\nmid a$, we must
have $h\mid p_2$ and $\tfrac{a}{c}\tfrac{p_2}{h}\in \S$. There exists then,
by Proposition~\ref{prop:barra}\emph{(i)}, $s_2\in \S$ such that $p_2=h
s_2\overline{\bigl(\tfrac{a}{c}\bigr)}$. This gives us a description of
homology: it is the free $\S$-module of rank~$2$ generated by the classes
of $y|\hY-x|\hX$ and
$\N(\tfrac{a}{c})h|\hH+\sigma(\tfrac{a'}{c}\overline{\tfrac{a}{c}}h)x|\hX$.

\end{itemize}

\item[Homology at $E_1^{2,0}$.]
Let $u\in E_1^{2,0}$, so in fact $u\in E_0^{2,0}$ and $\delta(u)=0$. In
view of~\eqref{eq:y-1}, there exists $p\in k[h]$ such that $u=yp|\hYH +
p\sigma(a')|\hYX +px|\hHX$.

The element $u$ is a boundary if there exist  $f_1$, $f_2\in k[h]$ such
that $u=d(y f_1|\hY + \tfrac{a}{c}f_2 |\hH +  (\sigma(\tfrac{a'}{c}
f_2)-f_1)x|\hX)$ or, making this explicit,
  \begin{gather*}
  p=(\sigma-q)(\tfrac{a}{c}f_2), \\
  \sigma(a')p= D_q(\tfrac{aa'}{c}f_2)- \tfrac{a}{c}f_2(q\sigma(a')-a').
\end{gather*}
The second equation follows from the first, and we conclude that $u$ is a
boundary if and only if $p\in\im\psi_{a/c,1}$ with $\psi_{a/c,1}$ defined
as in Lemma~\ref{lema:coker}. In other words, there is an isomorphism
$E_2^{2,0}\cong\coker\psi_{a/c,1}$. We have two cases:

\begin{itemize}[label=\textbullet]

\item First, suppose that \emph{$q$ is not a root of $1$}. If
$\deg(\tfrac{a}{c})>1$, then
$\deg\psi_{a/c,1}(f)=\deg(\tfrac{a}{c})+\deg(f)$ for $f\in k[h]\setminus0$.
It follows then that $\coker\psi_{a/c,1}$ is freely spanned by the classes
of $1$, $h$, \dots, $h^{N-M-1}$, because $\im\psi_{a/c,1}$ is spanned by a
set of polynomials of each degree greater or equal to $N-M$. We conclude
that $\dim(\coker(\psi_{a/c,1}))=N-M$.

On the other hand, if $\deg(\tfrac{a}{c})=1$, we have
$\deg\psi_{a/c,1}(f)=1+\deg(f)$ for all non-constant $f\in k[h]$ and
$\deg\psi_{a/c,1}(f)=0$ for $f\in k\setminus0$, so that the cokernel is
freely spanned by the class of $h$. In particular,
$\dim\coker(\psi_{a/c,1})=1=N-M$.

\item Suppose now that \emph{$q$ is a root of $1$}. We computed the
dimension of $\coker\psi_{a/c,1}$ in Lemma~\ref{lema:coker}, so that the
the dimension of $E^{2,0}_{2\wt0}$ is $\eta(a/c)$, as claimed in the
statement of the proposition. \qedhere

\end{itemize}

\end{itemize}
\end{proof}

\begin{Corollary}
If $q$ is a root of unity, then there are isomorphisms of $\S$-modules
  \begin{gather*}
  HH^p(A)_\wt0 \cong 
        \begin{cases}
        \S, & \text{if $p=0$;}  \\
        \S \oplus \S, & \text{if $p=1$;} \\
        \S \oplus k^{\eta(a/c)} \oplus k[h]/(c), & \text{if $p=2$.}
        \end{cases}
  \end{gather*}
If, on the other hand, $q$ has infinite order,
  \begin{gather*}
  HH^p(A)_\wt0 \cong 
        \begin{cases}
        k, & \text{if $p=0$;}  \\
        k, & \text{if $p=1$;} \\
        k^{N-M} \oplus k^M, & \text{if $p=2$.} 
        \end{cases}
  \end{gather*}
\end{Corollary}

\begin{proof}
This follows from the proposition and the convergence of the spectral
sequence.
\end{proof}

\begin{Remark}
In the computation of the Hochschild cohomology the fact that $a(0)\neq 0$
is only used in the proof of the Proposition~\ref{prop:co-e2-0}. In the
case when $q$ is not a root of $1$, using an analogous reasoning one can
prove that if $a(0)=0$ and $a\neq h^N$ then the same result holds. If
instead $a=h^N$ then
  \begin{gather*}
  E_{2\wt0}^{p,0} \cong 
        \begin{cases}
        k, & \text{if $p=0$;}  \\
        k^2, & \text{if $p=1$;} \\
        k^{N-M+1}, & \text{if $p=2$.}
        \end{cases}
  \end{gather*}
On the other hand, if $q$ is a root of $1$ then
\begin{gather*}
  E_{2\wt0}^{p,0} \cong 
        \begin{cases}
        \S, & \text{if $p=0$;}  \\
        \S \oplus \S, & \text{if $p=1$;} \\
        \S \oplus k^{\eta(a/(ch))+1} & \text{if $p=2$.}
        \end{cases}
\end{gather*}
This difference is to be expected because, for example, when $a=h^N$ we have gradings on
$A$ such that $\deg h = 1$ and $\deg x+\deg y = N$. The eulerian derivation induced by
one of these gradings is a non zero class in $HH^1(A)$, which is not cohomologous
to the induced by the weight.
\end{Remark}

\begin{Proposition}
Let $r\neq0$. According to whether $r$ is regular or not, there are
isomorphisms of $\S$-modules
  \begin{gather*}
  E_{2\wt r}^{p,0} \cong
                \begin{cases}
                0, & \text{if $p=0$;} \\
                0, & \text{if $p=1$;}  \\
                0, & \text{if $p=2$.}
                \end{cases}
   \\
\shortintertext{or}
  E_{2\wt r}^{p,0} \cong
                \begin{cases}
                \S, & \text{if $p=0$;} \\
                \S\oplus\S, & \text{if $p=1$;}  \\
                \S, & \text{if $p=2$.}
                \end{cases}
  \end{gather*}
\end{Proposition}

\begin{proof}
\begin{itemize}[fullwidth, font=\bfseries]

\item[Homology at $E_{1\wt r}^{0,0}$.]
Let $u\in E_{0\wt r}^{0,0}$, so that $u=y^rp$ for some $p\in k[h]$. Since
  \begin{equation} \label{eq:wer}
  d(u) = y^{r+1}(\sigma(p)-p)|\hY
       + (1-q^r)ph|\hH
       + y^{r-1}(a\sigma^{-1}(p)-\sigma^{r}(a)p)|\hX,
  \end{equation}
$u$ is a non zero cycle if and only if $r$ is a singular weight and $p\in \S$.

\item[Homology at $E_{1\wt r}^{1,0}$.]
If $u\in E_{1\wt r}^{1,0}$, then there exist $p_1$, $p_2$, $p_3\in k[h]$
such that $u=y^{r+1}p_1|\hY+y^{r}p_2|\hH+y^{r-1}p_3|\hX$ and $\delta(u)=0$.
This condition implies immediately, using \eqref{eq:cdelta-y},
\eqref{eq:cdelta-h} and \eqref{eq:cdelta-x}, that $p_3=p_2\sum_i \alpha_i
\q{i}{q^r} h^{i-1}-a\sigma^{-1}(p_1)$. Let us suppose now that $d(u)=0$.
\begin{itemize}[label=\textbullet]

\item If \emph{$r$ is regular}, we can replace $u$ by
$u-d((1-q^r)^{-1}y^r(p_2-p_2(0))/h)$ without changing its homology class,
and this amounts to assuming initially that $p_2\in k$. In that case, it is
easy to see that the coefficient of $\hYH$ in $d(u)$ is $y^{r+1}(q
(q^{r}-1)hp_1+(1-q)p_2)=0$ and, then, $p_1=p_2=0$. Similarly, looking at
the coefficient of~$\hHX$, we can conclude that $p_3=0$. 

\item If \emph{$r$ is singular}, there exist $b\in k[h]$ and $s_1\in S$
such that $p_1=\sigma(b)-b+s_1$; by replacing $u$ by $u-d(y^rb)$, which we
may do as it does not change the homology class, we may assume that
$p_1=s_1\in S$. Computing, we find that 
  \[
  d(u) = y^{r+1}(\sigma-q)(p_2)|\hYH 
       + y^r\sigma(a')(\sigma-q)(p_2)|\hYX 
       + y^{r-1}a(\sigma-q)(\sigma^{-1}(p_2))|\hHX, 
  \]
and it is clear that this vanishes exactly when $p_2\in h\S$. We see that
every element of $E_{2\wt r}^{1,0}$ is represented by an element in the
$\S$-submodule generated by the elements
  \begin{align*}
  &y^{r+1}|\hY - y^{r-1}a|\hX
  &&y^r h|\hH +y^{r-1}a'h|\hX.
  \end{align*}
Comparing with~\eqref{eq:wer}, it is easy to see that this submodule does
not contain non zero boundaries, so $E_{2\wt r}^{1,0}$ is $\S$-free of
rank~$2$.

\end{itemize}

\item[Homology at $E_{1\wt r}^{2,0}$.]
Let $u=y^{r+1}p_1|\hYH+y^{r}p_2|\hYX+y^{r-1}p_3|\hHX\in E_{1\wt r}^{2,0}$.
\begin{itemize}[label=\textbullet]

\item If \emph{$r$ is regular}, let $b_i=(p_i-p_i(0))(q(q^r-1)h)^{-1}$ for
$i\in\{1,3\}$. We may replace $u$ by $u-d(b_1|\hY+b_3|\hX)$, and a
computation using~\eqref{eq:cd-y} and~\eqref{eq:cd-x} shows that this means
that we can assume that $p_1$,~$p_3\in k$. Using now~\eqref{eq:cdelta-yh},
~\eqref{eq:cdelta-yx} and~\eqref{eq:cdelta-hx}, we easily see that
$\delta(u)=0$ if and only if $u=0$. It follows that in this case $E_{2\wt
r}^{2,0}=0$.

\item To finish, suppose next that \emph{$r$ is singular}. Since
  \[
  \delta(u) = y^{r+1}(\sigma(a')p_1-p_2)|\hY 
            + y^r(a\sigma^{-1}(p_1)-p_3)|\hH
            + y^{r-1}(a\sigma^{-1}(p_2)-a'p_3)|\hX = 0,
  \]
we see that $p_3=a\sigma^{-1}(p_1)$ and $p_2=\sigma(a')p_1$. If $b\in k[h]$
and $s\in S$ are such that $p_1=\sigma(b)-qb+hs$, we can replace $u$ by
$u-d(y^rb|\hH+y^{r-1}ba'|\hX)$, which is
  \[
  y^{r+1}hs_1|\hYH + y^r\sigma(a')hs_1|\hYX + y^{r-1}q^{-1}ahs_1|\hHX
  \]
without changing its class in $E_{2\wt r}^{2,0}$. No element of this form
is in the image of~$d$, as one can see by looking at the coefficient
of~$\hYH$ in~\eqref{eq:cd-y}, ~\eqref{eq:cd-h} and~\eqref{eq:cd-x}, so we
can conclude that $E_{2\wt r}^{2,0}$ is a free $\S$-module generated by the
class of $y^{r+1}h|\hYH + y^r\sigma(a')h|\hYX + y^{r-1}q^{-1}ah|\hHX$.

\end{itemize}

\end{itemize}
\end{proof}



\vspace*{\stretch{1}}

\normalfont

\noindent Departamento de Matem\'atica,\\
Facultad de Ciencias Exactas y Naturales,\\
Universidad de Buenos Aires,\\
Ciudad Universitaria, Pabell\'on 1\\
1428, Buenos Aires, Argentina.

\bigskip

\noindent Email: \texttt{asolotar@dm.uba.ar}, \texttt{mariano@dm.uba.ar}, \texttt{qvivas@dm.uba.ar}

\end{document}